\documentclass[10pt]{amsart}
\usepackage[cp1250]{inputenc}
\usepackage{latexsym}
\usepackage{amsmath}
\usepackage{amssymb}
\usepackage{amsthm}
\usepackage{amsbsy}
\usepackage{multirow}

\usepackage{soul} 
\usepackage{xcolor}

\newtheorem{Theorem}{Theorem}[section]
\newtheorem{Lemma}{Lemma}[section]
\newtheorem{Proposition}{Proposition}[section]
\newtheorem{Corollary}{Corollary}[section]
\newtheorem{Remark}{Remark}[section]

\begin{document}

\title{Central Limit Theorems and Minimum-Contrast Estimators for Linear Stochastic Evolution Equations}

\author{Pavel K\v r\' i\v z}
\address{University of Chemistry and Technology,  Prague, Department of Mathematics, Technick\' a 5, Prague 6, Czech Republic}
\email{pavel.kriz@vscht.cz}

\author{Bohdan Maslowski}
\address{Charles University in Prague, Faculty of Mathematics and Physics, Sokolovsk\' a 83, Prague 8, Czech Republic}
\email{maslow@karlin.mff.cuni.cz}
\thanks{The work of B. Maslowski was partially supported by the GACR Grant no. 15-08819S}



\date{28 May 2018}


\begin{abstract}
Central limit theorems and asymptotic properties of the minimum-contrast estimators of the drift parameter in linear stochastic evolution equations driven by fractional Brownian motion are studied. Both singular ($H < \frac{1}{2})$ and regular ($H > \frac{1}{2})$ types of fractional Brownian motion are considered. Strong consistency is achieved by ergodicity of the stationary solution. The fundamental tool for the limit theorems and asymptotic normality (shown for Hurst parameter $H < \frac{3}{4}$) is the so-called $4^{th}$ moment theorem considered on the second Wiener chaos. This technique provides also the Berry-Esseen-type bounds for the speed of the convergence. The general results are illustrated for parabolic equations with distributed and pointwise fractional noises.
\end{abstract}

\maketitle

\section{Introduction}

Estimation of the drift parameter in linear stochastic PDEs with additive noise is a problem that is both well-motivated by practical needs and theoretically challenging. Two pioneering works on this topic -- \cite{Kosky-Loges} by Koski, Akademi and Loges; and \cite{Huebner-Rozovskii1995} by Huebner and Rozovskii -- both considered Wiener process as the source of noise (i.e. the white noise in time). In \cite{Kosky-Loges} the minimum contrast (MC) estimator of the drift parameter was derived and its time asymptotics were studied (namely strong consistency and asymptotic normality). The paper \cite{Huebner-Rozovskii1995} constitutes the maximum likelihood estimator (MLE) of the drift parameter and studies conditions for strong consistency and asymptotic normality with increasing number of observed dimensions (the space asymptotics). Several other works dealing with estimation in the white noise setting followed. \\

The literature on parameter estimation in the drift term for SPDEs driven by a \textbf{fractional} Brownian motion (fBm), which is capable to generate noise coloured in time, is rather limited. The work \cite{Cialenco-Lototsky-Pospisil2009} develops the space asymptotics for the MLE in the fBm setting, but only in regular case when the Hurst parameter $H\geq \frac{1}{2}$. In \cite{MaPo-Ergo} the strong consistency of the MC estimator is proved considering fBm with trace-class covariance operator as the driving noise. Another contribution in this direction are the papers \cite{Maslowski-Tudor2013} and \cite{Balde-Sebayi-Tudor2018} dealing with the least squares estimators constructed from the one-dimensional projection of the mild solution to the linear SPDEs driven by fBm  and integrated fBm, respectively. Both works assume the Hurst parameter $1/2 \leq H \leq 3/4$. To authors' best knowledge, no proof of asymptotic normality of the MC estimator in the fBm setting has been published so far.\\


The strong consistency of the estimator follows from the appropriate strong law of large numbers (in the case of fBm-driven equations see e.g.  \cite{MaPo-Ergo} or \cite{Maslowski-Tudor2013}) while for the proof of asymptotic normality a central limit theorem-type result is needed, which may be viewed as a more challenging problem in the present case. Both problems are, of course, of independent interest.\\

However, recently developed theory of the Malliavin calculus and the Stein's method (widely known as the $4^{th}$ moment theorems) provides a very powerful tool for studying limit theorems of the Gaussian-subordinate sequences/processes, and, consequently, of the MC estimators. The idea of the $4^{th}$ moment theorem was first presented in \cite{Nualart-Pecatti2005} and further developed in many papers (see for example \cite{Bierme_et_al 2012} or \cite{Nourdin Peccati 2015} and references therein). \\

Applications of the Malliavin calculus to study asymptotic normality of the MC estimators of the drift parameters in one-dimensional linear SDEs driven by fBm were first presented in \cite{Hu-Nualart2010} (the MC estimator is called "Alternative estimator" there), but it covers only regular case $\frac{1}{2} < H < \frac{3}{4}$ and continuous-time observations. The results were later generalized for arbitrary $H\in (0,1)$ and for both continuous-time and discrete-time observations (combination of increasing time-horizon and observation frequency considered) in \cite{Hu-Nualart-Zhou2017}.  It was demonstrated there that asymptotic normality is violated for $H > \frac{3}{4}$. Similar approach used in more general setting of the fractional Vasicek model was presented in \cite{Xiao-Yu2018_1} and \cite{Xiao-Yu2018_2}, where both ergodic and non-ergodic cases are considered. Recently, the $4^{th}$ moment theorem was successfully utilized to demonstrate not only asymptotic normality, but also to establish the speed of the convergence to the normal distribution (Berry-Esseen-type of bounds) of the MC estimator of the drift parameter in one-dimensional SDEs driven by fBm - see \cite{SebaiyViens2016} for discrete-time observations with increasing time-horizon and fixed mesh or \cite{Sottinen-Vitasaary} for continuous-time observations or discrete-time observations with combination of increasing time-horizon and observation frequency.\\

In this paper, we study the asymptotic properties (in time) of the MC estimator of the drift parameter in infinite dimensional linear stochastic equations driven by a fBm (the solutions represent infinite dimensional fractional Ornstein-Uhlenbeck processes). We consider both continuous-time observations and discrete-time observations with increasing time-horizon and fixed mesh. We believe that most of the results concerning the speed of convergence in limit theorems are new also in the classical case of equations driven by a standard cylindrical Wiener process. The noise term in the equation may contain an unbounded operator, which makes the general results applicable to parabolic SPDEs with pointwise or boundary noise (see e.g. Example \ref{ex: parab. eq. point noise}). In section \ref{sec: Strong consistency}, we revise and generalize the proof of its strong consistency. Strong consistency without assuming that the covariance operator of the driving noise is trace-class is proved, which extends a result from \cite{MaPo-Ergo}. This generalization thus covers the basic equations with white noise in space (cylindrical fBm) and many others. Motivated by the recent development, we apply the techniques based on the $4^{th}$ moment theorem to prove asymptotic normality and construct Berry-Esseen-type bounds on the speed of the convergence. In section \ref{sec: Normal approximation for Gaussian-subordinated sequences/processes} we generalize the essential theory related to the $4^{th}$ moment theorem to the processes with continuous time and with values possibly in infinite dimensional Hilbert spaces. To authors' best knowledge, this approach has been studied only for real-valued random sequences with discrete time so-far. For continuously observed real-valued fractional Ornstein-Uhlenbeck processes, different approach based on the variances of the Malliavin derivatives of their Skorohod integrals was presented in \cite{Hu-Nualart-Zhou2017}.In section \ref{sec: Berry-Esseen bounds for sample moments} we apply this theory to construct the Berry-Esseen-type bounds for sample moments of the observed solution. These can be utilized for statistical inference on the drift parameter (such as confidence intervals or hypothesis testing) after reformulation in terms of the second moment. Finally, in section \ref{sec: Berry-Esseen bounds for estimators}, the Berry-Esseen bounds for MC estimators are constructed and asymptotic normality is proved. All the results on asymptotic normality and Berry-Esseen bounds assume that the Hurst parameter of the driving fBm fulfills $H < \frac{3}{4}$. Hence, both regular and singular cases are covered. Note that if $H > \frac{3}{4}$ the asymptotic normality cannot be expected  in general (cf. Remark \ref{rem: non-CLT}). Section \ref{sec: examples} contains two examples illustrating the general results: Stochastic linear parabolic equations with fractional distributed and pointwise noises, respectively (Examples \ref{ex: parab. eq. distr noise} and \ref{ex: parab. eq. point noise}).\\

Let us note that in practice it is often important to estimate the parameter $H$. In one-dimensional case it has been studied in many works, for example, in \cite{Brouste-Iacus2013} and references therein. These methods are, in principle, applicable also in the present case (for instance, by taking projections of solutions to finite-dimensional subspaces). However, we do not address this problem here.

\section{Basic setting and preliminaries}
Consider a linear stochastic evolution equation in a separable Hilbert space $\mathcal{V}$, which is driven by fractional Brownian motion in a separable Hilbert space $U$:
\begin{equation} \label{eq: SPDE}
\begin{aligned}
dX(t) &= \alpha A X(t) dt + \Phi dB^{H}(t), \quad t >0\\
\end{aligned}
\end{equation}
\begin{equation} \label{eq: init_Cond}
 X(0) = X_0.
\end{equation}

In this equation, $\alpha > 0$ is an \textbf{unknown} parameter and the linear operator $A: Dom(A) \subset \mathcal{V} \mapsto \mathcal{V}$ generates an analytic semigroup $(S(t), t \geq 0)$. Taking $\beta >0$ large enough so that the operator $\beta I-A$ is strictly positive, denote by $D_A^{\delta}$ the domain of the fractional power $(\beta I-A)^{\delta},\ \delta >0$, equipped with the graph norm $|\cdot |_{D_A^{\delta}} = |(\beta I-A)^{\delta}\cdot |_{\mathcal{V}}$ (in the sequel, $\beta$ is supposed to be fixed). Furthermore, let $(B^{H}(t), t \in \mathbb{R})$ be a standard (two-sided) cylindrical fractional Brownian motion on $U$ with Hurst parameter $H \in (0,1)$, defined on a complete probability space $(\Omega,\mathcal{F},\mathbb{P})$. The initial condition $X_0$ is assumed to be a $\mathcal{V}$-valued random variable such that $\mathbb{E}|X_0|^2_{\mathcal{V}}<\infty$. The noise term satisfies the following condition:\\
\begin{equation} \tag{A0}\label{eq: A0}
\begin{aligned}
& \Phi: U\supset Dom(\Phi) \to (Dom(A^*))', \text{ the dual of } Dom(A^*)  \text{ with respect to }\\
& \text {the topology of } \mathcal{V},\text{ and } (\beta I - A)^{\epsilon - 1} \Phi \in \mathcal{L} (U,\mathcal{V)} \text{ for a given } \epsilon \in (0,1].
\end{aligned}
\end{equation}

Notice that in the simplest case when \eqref{eq: A0} is satisfied with $\epsilon =1$ (equivalently $\Phi \in \mathcal{L}(U,\mathcal{V})$) and $H> 1/2$ the (strongly continuous) semigroup $(S(t), t \geq 0)$ need not be analytic. This is the most usual case which holds true in standard examples (for instance, stochastic parabolic PDE with distributed fractional noise, cf. Example \ref{ex: parab. eq. distr noise}). On the other hand, in some situations (stochastic parabolic PDE with pointwise or boundary fractional noise, cf. Example \ref{ex: parab. eq. point noise}) the value of $\epsilon$ must be chosen strictly smaller than one (see e.g. \cite{Maslowski_boundary pointwise} for more general results).\\

Obviously, $\alpha A$ generates an analytic semigroup $(S_\alpha(t), t \geq 0)$, where $S_\alpha(t) = S(\alpha t)$.\\

Existence of the $\mathcal{V}$-valued mild solution is established in the following proposition (for proofs, see papers \cite{DuncanMaslowski2002} for $H > \frac{1}{2}$ and \cite{DuncanMaslowski2006} for $H < \frac{1}{2}$).

\begin{Proposition}\label{prop: existence of solut}
Assume \eqref{eq: A0} and
\begin{equation} \tag{A1}\label{eq: A1}
\begin{aligned}
|S(t)\Phi|_{\mathcal{L}_2}\leq ct^{-\gamma}, \quad \forall t \in (0,T],
\end{aligned}
\end{equation}
for some  $T>0, c>0 \text{ and } \gamma \in [0,H)$, where $|.|_{\mathcal{L}_2}$ denotes the Hilbert-Schmidt norm on $\mathcal{V}$. Then
\begin{equation} \label{eq: general_solut}
X(t) = S_\alpha(t) X_0 + \int_{0}^{t}S_\alpha(t-u)\Phi dB^H(u)
\end{equation}
is a well-defined $\mathcal{V}$-valued process with continuous paths in $\mathcal{C}([0,T],\mathcal{V})$ and is called the mild solution to the equation \eqref{eq: SPDE} satisfying the initial condition \eqref{eq: init_Cond}.
\end{Proposition}

The following theorem (established in \cite{MaPo-Ergo} and \cite{MaPo-Ergo2}) ensures existence and ergodicity of stationary solutions.

\begin{Proposition}\label{prop: existence of ergodic solut}
Assume \eqref{eq: A0}, \eqref{eq: A1} and
\begin{equation} \tag{A2}\label{eq: A2}
\begin{aligned}
|S(t)|_{\mathcal{L}}\leq M e^{-\rho t}, \quad \forall t > 0,\\
\end{aligned}
\end{equation}
for some constants $ M>0 \text{ and } \rho > 0$, where $|.|_{\mathcal{L}}$ denotes the operator norm. Then there exists a strictly stationary continuous solution to $\eqref{eq: SPDE}$, i.e. there exists an initial value $\tilde{X}$ (a $\mathcal{V}$-valued random variable) such that the solution
\begin{equation} \label{eq: stat_solut}
Z(t) = S_\alpha(t) \tilde{X} + \int_{0}^{t}S_{\alpha}(t-u)\Phi dB^H(u), \quad t \geq 0,
\end{equation}
is a strictly stationary process with continuous paths.\\

\noindent Moreover, under assumptions \eqref{eq: A0}- \eqref{eq: A2} the strictly stationary solution $(Z(t),t\geq 0)$ is ergodic, i.e. for any measurable functional $\varrho : \mathcal{V} \to \mathbb{R}$ satisfying $\mathbb{E}|\varrho(\tilde{X})|<\infty$, the following holds:
\begin{equation} \label{eq: Ergo_stat}
\begin{aligned}
&\lim_{T\to \infty} \frac{1}{T} \int_{0}^{T}\varrho(Z(t))dt = \int_{\mathcal{V}} \varrho(x) \mu(dx), \quad \text{ a.s.},\\
&\lim_{n\to \infty} \frac{1}{n} \sum_{i=1}^{n}\varrho(Z(i)) = \int_{\mathcal{V}} \varrho(x) \mu(dx), \quad \text{a.s.}
\end{aligned}
\end{equation}
where
\[
\mu = \text{Law}(\tilde{X}) = \text{Law}(Z(t)) = N(0,Q^{\alpha}_{\infty}), \quad \forall t\geq 0.
\]
\end{Proposition}
\begin{proof}
See proofs of Theorems 3.1, 4.6 in \cite{MaPo-Ergo} and  Theorems 3.1, 3.2 in \cite{MaPo-Ergo2}. Note that the discrete-time ergodicity can be shown by same means as continuous-time ergodicity.
\end{proof}

In \cite{MaPo-Ergo} it is shown that
\begin{equation}\label{eq: Q^aplha}
Q^{\alpha}_{\infty} = \frac{1}{\alpha^{2H}}Q_{\infty},
\end{equation}
where $Q_{\infty}$ is the covariance operator of the invariant measure in the case $\alpha = 1$.
The minimum contrast estimator is based on the ergodic  behaviour of the  solutions to the equation $\eqref{eq: SPDE}$. The ergodicity of the stationary solution implies the following (almost sure) convergence of the sample second moments:
\[
\lim_{T\to \infty} \frac{1}{T} \int_{0}^{T}|Z(t)|^2_\mathcal{V} dt = \int_{\mathcal{V}} |x|^2_\mathcal{V} \mu(dx) = \mathbb{E}|Z(t)|^2_\mathcal{V} = {\rm Tr}(Q^{\alpha}_{\infty}) = \frac{1}{\alpha^{2H}} {\rm Tr}(Q_{\infty}).
\]
Similar behaviour of the (non-stationary) solution $(X(t), t \geq 0)$ for large $t$ motivates the construction of the minimum-contrast estimator of the parameter $\alpha$. Assume that
\begin{equation}\label{eq: Phi non zero}
\Phi \neq 0,
\end{equation}
to ensure that ${\rm Tr}(Q_{\infty}) \neq 0$ and the following minimum contrast estimator is well defined:

\begin{equation} \label{eq: MC_est_continuous}
\hat{\alpha}_{T} := \biggl( \frac{\frac{1}{T}\int_{0}^{T}|X(t)|_{\mathcal{V}}^{2}dt}{{\rm Tr} (Q_{\infty})} \biggr)^{-\frac{1}{2H}}.
\end{equation}
Similarly, for observations in discrete time instants with fixed step sizes define
\begin{equation} \label{eq: MC_est_discrete}
\check{\alpha}_{n} := \biggl( \frac{\frac{1}{n}\sum_{i=1}^{n}|X(i)|_{\mathcal{V}}^{2}}{{\rm Tr} (Q_{\infty})} \biggr)^{-\frac{1}{2H}}.
\end{equation}
Obviously, in the degenerate case $\Phi = 0$ the stationary solution is constantly zero and the parameter is not identifiable.

We may also consider the case when only a finite-dimensional projection of the solution $(X(t), t \geq 0)$ is observed, for instance, the process $(\langle X(t),w \rangle_\mathcal{V}, t \geq 0)$ for a given vector $w \in \mathcal{V}$, such that
\begin{equation}
\langle Q_{\infty} w,w\rangle_\mathcal{V} \neq 0.
\end{equation}
We obtain the estimators
\begin{equation} \label{eq: proj_est_continuous}
\tilde{\alpha}_{T} := \biggl( \frac{\frac{1}{T}\int_{0}^{T}|\langle X(t),w \rangle_\mathcal{V}|^{2}dt}{\langle Q_{\infty} w,w\rangle_\mathcal{V}} \biggr)^{-\frac{1}{2H}},
\end{equation}
and for observations in discrete time instants with fixed step sizes
\begin{equation} \label{eq: proj_est_discrete}
\bar{\alpha}_{n} := \biggl( \frac{\frac{1}{n}\sum_{i=1}^{n}|\langle X(i),w \rangle_\mathcal{V}|^{2}}{\langle Q_{\infty} w,w\rangle_\mathcal{V}} \biggr)^{-\frac{1}{2H}}.
\end{equation}

\section{Strong consistency}\label{sec: Strong consistency}
Note that in \cite{MaPo-Ergo} the strong consistency of the estimators \eqref{eq: MC_est_continuous} and \eqref{eq: proj_est_continuous} is proved for fractional Brownian motion ($H \in (0,1)$), but only for $\Phi$ being Hilbert-Schmidt. The approach below, based on Taylor's approximation, enables to show the strong consistency for a general linear operator $\Phi$ satisfying \eqref{eq: A0}. In the following exposition, we shall focus on the continuous-time case. The discrete-time case is analogous.\\

Let us first state a simple technical lemma, the proof of which is obvious:
\begin{Lemma} \label{lem: means_to_zero}
Consider continuous functions $g,h:[0,+\infty) \mapsto \mathbb{R}$ such that
\begin{itemize}
    \item $\frac{1}{T} \int_{0}^{T} |g(t)| dt \overset{T \to \infty}{\longrightarrow} K < \infty $; and
    \item $h(t) \overset{t \to \infty}{\longrightarrow} 0$ .
\end{itemize}
Then
\[
\frac{1}{T} \int_{0}^{T} g(t)h(t) dt \overset{T \to \infty}{\longrightarrow} 0.
\]
\end{Lemma}


To show the convergence of the sample second moment of the non-stationary solution, we utilize ergodicity of the stationary solution and the following theorem:

\begin{Theorem} \label{thm: ergo non_stat}
Let $X=(x_t, t \geq 0)$ and $Y=(y_t, t \geq 0)$ be two real-valued continuous random processes defined on a probability space $(\Omega, \mathcal{F}, \mathbb{P})$. Let $Y$ be strictly stationary and ergodic. Further let
\begin{equation}\label{eq: xt to yt}
|x_t-y_t| \overset{t \to \infty}{\longrightarrow} 0 \quad \text{a.s.}
\end{equation}
Consider a smooth function $g \in \mathcal{C}^{p}(\mathbb{R})$ for some integer $p \geq 1$. Denote its $k$-th derivative by $g^{(k)}$ and assume:

\begin{itemize}
    \item $\mathbb{E} |g^{(k)}(y_0) | < \infty; \; k = 0,1,..,p;$ and
    \item $g^{(p)}$ is globally Lipschitz. 
\end{itemize}
Then
\begin{equation}\label{eq: ergo non_stat 1}
\frac{1}{T} \int_{0}^{T} g(x_t) dt \overset{T \to \infty}{\longrightarrow} \mathbb{E}g(y_0) \quad \text{a.s.}
\end{equation}
\end{Theorem}

\begin{proof}
Ergodicity of $Y$ implies $\frac{1}{T} \int_{0}^{T} g(y_t) dt \overset{T \to \infty}{\longrightarrow} \mathbb{E}g(y_0) \quad \text{a.s.}$. It suffices to prove
\[
\biggl|\frac{1}{T} \int_{0}^{T}\biggl(g(x_t)- g(y_t)\biggr) dt \biggr| \overset{T \to \infty}{\longrightarrow} 0 \quad \text{a.s.}
\]
Now fix $\omega \in \Omega$ and apply Taylor's approximation with Lagrange remainder:
\begin{equation*}
\begin{aligned}
&\biggl|\frac{1}{T} \int_{0}^{T}\biggl(g(x_t(\omega))- g(y_t(\omega))\biggr) dt \biggr| \leq \frac{1}{T} \int_{0}^{T}\biggl|g(x_t(\omega))- g(y_t(\omega))\biggr| dt \\
& = \frac{1}{T} \int_{0}^{T}\biggl| \sum_{k=1}^{p-1} \frac{g^{(k)}(y_t(\omega))}{k!}(x_t(\omega)- y_t(\omega))^{k} +\frac{g^{(p)}(\eta_t(\omega))}{p!}(x_t(\omega)- y_t(\omega))^{p} \biggr| dt  \\
&\leq \sum_{k=1}^{p-1} \frac{1}{T} \int_{0}^{T}|g^{(k)}(y_t(\omega))| \frac{|x_t(\omega)- y_t(\omega)|^{k}}{k!} dt  + \frac{1}{T} \int_{0}^{T}|g^{(p)}(\eta_t(\omega))| \frac{|x_t(\omega)- y_t(\omega)|^{p}}{p!}dt,
\end{aligned}
\end{equation*}
where $\eta_t$ is a random point between $x_t$ and $y_t$.\\
Firstly, for almost all (a.a.) $\omega$ we have $\frac{1}{T} \int_{0}^{T}|g^{(k)}(y_t(\omega))|dt \overset{T \to \infty}{\longrightarrow} \mathbb{E}|g^{(k)}(y_0)| < \infty$ by ergodicity of $Y$ and $\frac{|x_t(\omega)- y_t(\omega)|^{k}}{k!}\overset{t \to \infty}{\longrightarrow} 0$ by \eqref{eq: xt to yt}. Hence, via Lemma \ref{lem: means_to_zero}, we obtain for a.a $\omega$
\[\frac{1}{T} \int_{0}^{T}|g^{(k)}(y_t(\omega))| \frac{|x_t(\omega)- y_t(\omega)|^{k}}{k!} dt \overset{T \to \infty}{\longrightarrow} 0, \text{ for $k=1,...,p-1$.}\]
Secondly,
\[\lim_{T\to \infty}\frac{1}{T} \int_{0}^{T}|g^{(p)}(\eta_t(\omega))|dt =\lim_{T\to \infty} \frac{1}{T} \int_{0}^{T}|(g^{(p)}(\eta_t(\omega))-g^{(p)}(y_t(\omega)))+g^{(p)}(y_t(\omega))| dt.\] However, the Lipschitz condition implies
\[|g^{(p)}(\eta_t(\omega))-g^{(p)}(y_t(\omega))|\leq C|\eta_t(\omega)-y_t(\omega)|\leq C|x_t(\omega)-y_t(\omega)| \overset{t \to \infty}{\longrightarrow} 0 \text{ for a.a. $\omega$.}\]
Hence, we obtain for a.a. $\omega$:
\[
\lim_{T\to \infty} \frac{1}{T} \int_{0}^{T}|(g^{(p)}(\eta_t(\omega))| dt = \lim_{T\to \infty} \frac{1}{T} \int_{0}^{T}|g^{(p)}(y_t(\omega))| dt = \mathbb{E}|g^{(p)}(y_0)| < \infty\]
by ergodicity of $Y$. If we combine this with the a.s. convergence
\[
\frac{|x_t(\omega)- y_t(\omega)|^{p}}{p!}\overset{t \to \infty}{\longrightarrow}0,
\]
 Lemma \ref{lem: means_to_zero} yields
\[
\lim_{T\to \infty}\frac{1}{T} \int_{0}^{T}|g^{(p)}(\eta_t(\omega))| \frac{|x_t(\omega)- y_t(\omega)|^{p}}{p!}dt = 0 \text{ for a.a. $\omega$.}
\]
\end{proof}

\begin{Corollary}\label{cor: convergence of non-stat moments}
Let \eqref{eq: A0}, \eqref{eq: A1} and \eqref{eq: A2} are satisfied and let $(X(t),t \geq 0)$ be the solution to (\ref{eq: SPDE}) - (\ref{eq: init_Cond}) with $X_0$ such that $\mathbb{E}|X_0|^2_{\mathcal{V}}<\infty$.
For any integer $p \geq 1$ we have that
\begin{equation}\label{eq: nonstat MC ergo 1}
\frac{1}{T} \int_{0}^{T} |X(t)|_{\mathcal{V}}^{p} dt \overset{T \to \infty}{\longrightarrow} \int_{\mathcal{V}} |y|_{\mathcal{V}}^{p} \mu(dy) \text{ a.s.}.
\end{equation}
Moreover, for any $w \in \mathcal{V}$
\begin{equation}\label{eq: nonstat proj ergo 1}
\frac{1}{T} \int_{0}^{T} |\langle X(t),w \rangle_{\mathcal{V}}|^{p} dt \overset{T \to \infty}{\longrightarrow} \int_{\mathcal{V}} |\langle y,w \rangle_{\mathcal{V}}|^{p} \mu(dy) \text{ a.s.}.
\end{equation}
\end{Corollary}

\begin{proof}
Starting with \eqref{eq: nonstat MC ergo 1}, consider $(Z(t),t \geq 0)$ the strictly stationary solution to \eqref{eq: SPDE} and set $y_t = |Z(t)|_\mathcal{V}$ and $x_t = |X(t)|_\mathcal{V}$. Obviously, $(y_t, t \geq 0)$ and $(x_t, t \geq 0)$ have continuous trajectories by continuity of processes $Z$ and $X$ (cf. Proposition \ref{prop: existence of solut} and Proposition \ref{prop: existence of ergodic solut}) and $(y_t, t \geq 0)$ is strictly stationary and ergodic (by stationarity and ergodicity of $Z$). Clearly $|x_t - y_t| \leq |S(t)|_{\mathcal{L}} |X(0)-Z(0)|_\mathcal{V} \to 0$ a.s..\\

For $g(x) = x^p$, in view of the Fernique theorem, we obtain
\begin{equation*}
\begin{aligned}
& \mathbb{E} |g^{(k)}(y_0)| = \mathbb{E}(c \, |Z(0)|^{p-k}_\mathcal{V}) < \infty; \; k = 0,1,..,p-1,
\end{aligned}
\end{equation*}
and $g^{(p-1)}(x) = p! x$ is globally Lipschitz. Therefore, the proof of (\ref{eq: nonstat MC ergo 1}) is completed in virtue of Theorem \ref{thm: ergo non_stat}.
\\

Proof of convergence \eqref{eq: nonstat proj ergo 1} runs similarly using processes $y_t = |\langle Z(t),w \rangle_\mathcal{V}|$ and $x_t = |\langle X(t),w \rangle_\mathcal{V}|$.
\end{proof}

\begin{Corollary} \label{cor: consistency}
Under the assumptions \eqref{eq: A0}, \eqref{eq: A1} and \eqref{eq: A2}, the estimators \eqref{eq: MC_est_continuous}, \eqref{eq: MC_est_discrete}, \eqref{eq: proj_est_continuous} and \eqref{eq: proj_est_discrete} are strongly consistent, i.e.
\begin{equation*}
\begin{aligned}
&\hat{\alpha}_{T}:= \biggl( \frac{\frac{1}{T}\int_{0}^{T}|X(t)|_{\mathcal{V}}^{2}dt}{{\rm Tr} (Q_{\infty})} \biggr)^{-\frac{1}{2H}} \overset{T \to \infty}{\longrightarrow} \alpha \quad \text{a.s.} ,\\
&\check{\alpha}_{n} := \biggl( \frac{\frac{1}{n}\sum_{i=1}^{n}|X(i)|_{\mathcal{V}}^{2}}{{\rm Tr} (Q_{\infty})} \biggr)^{-\frac{1}{2H}}\overset{n \to \infty}{\longrightarrow} \alpha \quad \text{a.s.} ,\\
&\tilde{\alpha}_{T} := \biggl( \frac{\frac{1}{T}\int_{0}^{T}|\langle X(t),w \rangle_\mathcal{V}|^{2}dt}{\langle Q_{\infty} w,w\rangle_\mathcal{V}} \biggr)^{-\frac{1}{2H}}\overset{T \to \infty}{\longrightarrow} \alpha \quad \text{a.s.} ,\\
&\bar{\alpha}_{n} := \biggl( \frac{\frac{1}{n}\sum_{i=1}^{n}|\langle X(i),w \rangle_\mathcal{V}|^{2}}{\langle Q_{\infty} v,v\rangle_\mathcal{V}} \biggr)^{-\frac{1}{2H}}\overset{n \to \infty}{\longrightarrow} \alpha \quad \text{a.s.}.
\end{aligned}
\end{equation*}
\end{Corollary}

\begin{proof}
The continuous-time versions of the statements follow immediately from Corollary \ref{cor: convergence of non-stat moments} by setting $p=2$. The long-span asymptotics with fixed time step in the discrete-time versions may be shown analogously using discrete version of ergodicity.
\end{proof}


\section{Normal approximation for Gaussian-subordinated sequences/processes}\label{sec: Normal approximation for Gaussian-subordinated sequences/processes}

\subsection{Preliminaries}

The fundamental tool for proving asymptotic normality (as well as for assessing the speed of convergence) will be the celebrated $4^{th}$ moment theorem (see \cite{Nourdin Peccati 2015}):

\begin{Proposition}\label{prop: optimal 4th moment}
Consider an isonormal Gaussian process $\mathbb{X}$ on a separable Hilbert space $\mathcal{H}$. Let $(F_n: n \in \mathbb{N})$ be a sequence of random variables belonging to the $q$-th chaos of $\mathbb{X}$ with $\mathbb{E}F^2_n = 1$ and consider a normally distributed random variable $N \sim \mathcal{N}(0,1)$. Then sequence  $F_n$ converges in distribution to $N$ if and only if the $4^{th}$ cumulants of $F_n$ ($\kappa_4(F_n)$) converge to zero, i.e.
\begin{equation}\label{eq:AN with 4th moment}
F_n \overset{d}{\to} N \quad \Longleftrightarrow \quad \mathbb{E}F^4_n - 3 \to 0.
\end{equation}

In this case $\mathbb{E}F^3_n \to 0$ and there exist positive finite constants $0< c < C < \infty$ (which may depend on the sequence $(F_n)$, but not on $n$) such that
\begin{equation}\label{eq:optimum B-E bounds}
c \; M(F_n) \leq d_{TV}(F_n,N) \leq C \; M(F_n),
\end{equation}
where $d_{TV}$ denotes the total-variation distance and
\[
M(F_n) = \max \{|\mathbb{E}F^3_n|,|\mathbb{E}F^4_n - 3| \} = \max \{|\kappa_3(F_n)|,|\kappa_4(F_n)|\}
\]
represents the optimum bound of Berry-Esseen type for $F_n$ expressed in terms of the $3^{rd}$ and the $4^{th}$ cumulants.

\end{Proposition}

\begin{Remark}\label{rem: 4th moment for d_W }
It directly follows from the proof of Proposition \ref{prop: optimal 4th moment} (see the proof of Theorem 1.2 in \cite{Nourdin Peccati 2015}) that the corresponding upper bound
\[
d_{TV}(F_T,N) \leq C \max \{|\kappa_3(F_T)|,|\kappa_4(F_T)|\}
\]
holds also for a random process $(F_T: T>0)$ from the $q$-th Wiener chaos. \\
Furthermore, the upper bound in \eqref{eq:optimum B-E bounds} also holds for the Wasserstein distance:
\[
d_{W}(F_n,N) \leq C \; \max \{|\kappa_3(F_n)|,|\kappa_4(F_n)|\},
\]
possibly with different constant $C$. To show this inequality, we can employ conditional expectation in the proof of formula $(3.35)$ in \cite{Nourdin Peccati 2009 PTRF} (cf. Theorem 3.1 therein) to formulate Proposition 2.4 from \cite{Nourdin Peccati 2015} in terms of the Wasserstein distance and apply the proof of the upper bound in Theorem 1.2 therein.
\end{Remark}


In order to use the $4^{th}$ moment theorem, we will need to calculate the $3^{rd}$ and the $4^{th}$ cumulants of the sample second moments of the observations.

\subsection{$\infty$-dimensional Gaussian-subordinated sequences}\label{subsec: infty D sequences}
Following the approach for one-dimensional Gaussian-subordinated sequences presented in the paper \cite{Bierme_et_al 2012}, we derive corresponding results for Gaussian sequences with values in an infinite-dimensional separable Hilbert space $\mathcal{V}$ with an orthonormal basis $\{e_{k}\}_{k=1}^{\infty}$.\\

Let $(Z_i: i \in \mathbb{Z})$ be a non-degenerate stationary $\mathcal{V}$-valued centered Gaussian sequence with a trace class covariance operator $Q \neq 0$, i.e. $Z_i \sim N(0,Q)$ and denote $Q(i-j) =
cov(Z_i,Z_j)$ the auto-covariance operators of the sequence, which means that
\[
\mathbb{E}(\langle Z_i,v\rangle_\mathcal{V} \; \langle Z_j,w\rangle_\mathcal{V} ) = \langle Q(i-j)v,w\rangle_\mathcal{V} \quad \forall v,w \in \mathcal{V}.
\]
Now define
\begin{equation}\label{eq: def V_n F_n - infty}
V_n :=  \frac{1}{\sqrt{n}} \sum_{i=1}^{n}(|Z_i|_\mathcal{V}^2 - {\rm Tr}(Q)), \quad s_n := \mathbb{E}V^2_n, \quad F_n := \frac{V_n}{\sqrt{s_n}}.
\end{equation}

We can express $s_n$ in terms of $Q$ (using the spectral decomposition and Isserlis theorem):
\begin{equation}\label{eq: s_n via Q}
s_n = 2\sum_{i=-(n-1)}^{n-1}\biggl(1 - \frac{|i|}{n} \biggr) |Q(i)|^2_{\mathcal{L}_2}.
\end{equation}

We need to define a Hilbert space with scalar product generated by the auto-covariance operators $Q(i)$. To this end, we have to assume positive definiteness of $Q$. This assumption, however, can be made without loss of generality. If $Q$ is  only positive-semidefinite, we can take an orthonormal basis $\{e_{k}\}_{k=1}^{\infty}$ of $\mathcal{V}$ consisting of eigenvectors of $Q$ and restrict ourselves on the subspace of $\mathcal{V}$ generated by those eigenvectors corresponding to positive eigenvalues (denote it by $\mathcal{V}_{pos}$). Clearly the projections of $Z_i$ onto the orthogonal complement of the subspace $\mathcal{V}_{pos}$ is (almost surely) zero. Hence, $Z_i \in \mathcal{V}_{pos}$ almost surely  and $Q$, restricted to $\mathcal{V}_{pos}$, is positive-definite.\\

The construction of the corresponding Hilbert space $\mathcal{H}$ and isonormal Gaussian process for Wiener chaos decomposition follows the approach from \cite{Nourdin Peccati 2012 book} (cf. Proposition 7.2.3). Consider the linear span of the abstract set $\mathcal{E}:= \{h_{i,v} : i \in \mathbb{Z}, v \in \mathcal{V} \} \subset \mathcal{H}$ endowed with the scalar product obtained by natural extension of the following binary operation on $\mathcal{E}$:
\[\langle h_{i,v}, h_{j,w} \rangle_{\mathcal{H}} := \langle Q(i-j)v,w\rangle_{\mathcal{V}}.\]
Taking the completion of this linear space with respect to the scalar product yields the separable Hilbert space $\mathcal{H}$. Separability of $\mathcal{H}$ follows from separability of $\mathcal{V}$. The appropriate isonormal Gaussian process $\mathbb{X}$ is defined as the $L^2$-isometric linear extension of the mapping
\[
\mathbb{X}(h_{i,v}) =  \langle Z_i,v\rangle_\mathcal{V}, \quad \forall h_{i,v} \in \mathcal{E}.
\]

Next step is to express $F_n$ as an element of the second Wiener chaos of $\mathbb{X}$. Consider $h_{i,e_k}$ the elements of $\mathcal{H}$. Then
\begin{equation}\label{eq: Fn via I_2}
\begin{aligned}
&F_n = \frac{1}{\sqrt{n \; s_n}}\sum_{i=1}^{n} \sum_{k=1}^{\infty} (\langle Z_i,e_k\rangle_\mathcal{V}^2 - \langle Q e_k,e_k\rangle_\mathcal{V})\\
& = \frac{1}{\sqrt{n \; s_n}}\sum_{i=1}^{n} \sum_{k=1}^{\infty} (\mathbb{X}^2(h_{i,e_k}) - \langle Q e_k,e_k\rangle_\mathcal{V}) \\
& = \frac{1}{\sqrt{n \; s_n}}\sum_{i=1}^{n} \underset{N \to \infty}{\text{$L^{2}$-lim}}\sum_{k=1}^{N} I_2(h_{i,e_k}^{\otimes 2})\\
& = I_2\biggl( \frac{1}{\sqrt{n \; s_n}} \sum_{i=1}^{n} \sum_{k=1}^{\infty} h_{i,e_k}^{\otimes 2} \biggr) = I_2(f_n) ,
\end{aligned}
\end{equation}
where $I_2$ is the isometric isomorphism between the space of symmetric tensor product $\mathcal{H}^{\odot 2}$ (equipped with the modified tensor norm $\sqrt{2}|.|_{\mathcal{H}^{\otimes 2}}$) and the second Wiener chaos of $\mathbb{X}$ (equipped with the $L^2$ norm).\\

The $3^{rd}$ and the $4^{th}$ cumulants of $F_n$ can be bounded above (see Lemma \ref{lem: 4  cumul bound infty D} in appendix for details):

\begin{equation}\label{eq: 3  cumul bound infty D}
\kappa_3(F_n) \leq \frac{C_1}{\sqrt{n}\; s_n^{3/2}} \biggl( \sum_{|i|<n} |Q(i)|^{3/2}_{\mathcal{L}_2} \biggr)^2,
\end{equation}
where $C_1$ is arbitrary constant greater than 8, and
\begin{equation}\label{eq: 4  cumul bound infty D}
\kappa_4(F_n) \leq \frac{C_2}{n\; s_n^2} \biggl( \sum_{|i|<n} |Q(i)|^{4/3}_{\mathcal{L}_2} \biggr)^3,
\end{equation}
where $C_2$ is a (sufficiently large) constant.

\begin{Remark}
Real-valued Gaussian-subordinated sequences are also covered by the theory above. The bounds \eqref{eq: 3  cumul bound infty D} and \eqref{eq: 4  cumul bound infty D} in this special case correspond to those in \cite{Bierme_et_al 2012}.
\end{Remark}

\subsection{$\infty$-dimensional Gaussian-subordinated processes}
Consider a non-degenerate centered stationary continuous Gaussian random process $(Z_t: t \in \mathbb{R})$ with values in a separable Hilbert space $\mathcal{V}$ with orthonormal basis $\{e_k\}_{k=1}^{\infty}$. Denote by $Q(t-s)$ the (trace-class) auto-covariance operators of the process (i.e. $Q(t-s) = cov(Z_t,Z_s)$) and $Q=Q(0)\neq 0$ the covariance operator of $Z_t$. Define the corresponding Gaussian subordinate processes:
\begin{equation}\label{eq: def V_T F_T - infty}
 V_T :=  \frac{1}{\sqrt{T}} \int_{0}^{T}(|Z_t|_{\mathcal{V}}^2  - {\rm Tr}(Q)) dt, \quad s_T := \mathbb{E}V^2_T, \quad F_T := \frac{V_T}{\sqrt{s_T}}.
\end{equation}

The spectral decomposition, the Isserlis' theorem and the change of variables within integrals yield:
\begin{equation}\label{eq: s_T via Q}
s_T = 2\int_{-T}^{T}\biggl(1 - \frac{|t|}{T} \biggr) |Q(t)|^2_{\mathcal{L}_2}  dt.
\end{equation}

To construct the appropriate Hilbert space $\mathcal{H}$ with the isonormal Gaussian process $\mathbb{X}$, start with a set $\mathcal{E}:= \{h_{t,v} : t \in \mathbb{R} , v \in \mathcal{V} \}$ and proceed analogously to previous subsection. Separability of $\mathcal{H}$ follows from separability of $\mathcal{V}$ and from $L^2$-continuity of process $(Z_t: t \in \mathbb{R})$.


To enable using results for random sequences, fix $T>0$ and consider the partition $\{t_i = i \frac{T}{n}: i = 0,...,n\}$ of the interval $[0,T]$. Further consider $V_n^{(T)}$ constructed from the stationary sequence $(Z_{t_1},Z_{t_2},...,Z_{t_n})$ with the auto-covariance operators $Q_n^{(T)}(i) = Q(i\; T/n)$.\\

The $L^2$ continuity of the real-valued random process $|Z_.|_{\mathcal{V}}^2 : t \mapsto |Z_t|_{\mathcal{V}}^2$ implies:
\[
V_T= \frac{1}{\sqrt{T}} \int_{0}^{T}(|Z_{t}|_{\mathcal{V}}^2 - {\rm Tr}(Q)) dt = \underset{n \to \infty}{\text{$L^{2}$-lim}} \frac{1}{\sqrt{T}} \sum_{i=1}^{n}(|Z_{t_i}|_{\mathcal{V}}^2 - {\rm Tr}(Q)) \frac{T}{n} =\underset{n \to \infty}{\text{$L^{2}$-lim}} \sqrt{\frac{T}{n}}V_n^{(T)}.
\]
Therefore
\[
s_T = \mathbb{E}V_T^2 = \lim_{n \to \infty} \mathbb{E}\biggl(\sqrt{\frac{T}{n}}V_n^{(T)}\biggr)^2 = \lim_{n \to \infty} \frac{T}{n}s_n^{(T)},
\]
\[
F_T = \underset{n \to \infty}{\text{$L^{2}$-lim}} \frac{\sqrt{\frac{T}{n}}V_n^{(T)}}{\sqrt{\frac{T}{n}s_n^{(T)}}} = \underset{n \to \infty}{\text{$L^{2}$-lim}} \frac{V_n^{(T)}}{\sqrt{s_n^{(T)}}}= \underset{n \to \infty}{\text{$L^{2}$-lim }} F_n^{(T)}.
\]
Since
\[
F_n^{(T)} = I_2\biggl( \frac{1}{\sqrt{n \; s_n^{(T)}}} \sum_{i=1}^{n} \sum_{k=1}^{\infty} h_{t_i,e_k}^{\otimes 2} \biggr)
\]
belongs to the second chaos, $F_T$ is in the second chaos as well, with
\[
F_T = I_2\biggl( \frac{1}{\sqrt{T \; s_T}} \int_{0}^{T} \sum_{k=1}^{\infty} h_{t,e_k}^{\otimes 2} dt \biggr).
\]

By hypercontractivity property on the second Wiener chaos, $\mathbb{E}|F_n^{(T)} - F_T|^p \to 0$ for any $p\geq 2$, which implies
\[ \kappa_3(F_T) = \lim_{n \to \infty} \kappa_3(F_n^{(T)}) \quad \text{and} \quad  \kappa_4(F_T) = \lim_{n \to \infty} \kappa_4(F_n^{(T)}).  \]

Now we can utilize the bounds from the previous subsection (note that $ t \mapsto |Q(t)|_{\mathcal{L}_2}$ is continuous):
\begin{equation}\label{eq: 3rd cumul infty-D continuous}
\begin{aligned}
&\kappa_3(F_T) = \lim_{n \to \infty} \kappa_3(F_n^{(T)})  \leq  \liminf_{n \to \infty} \frac{C_1}{\sqrt{n}\; (s_n^{(T)})^{3/2}} \biggl( \sum_{|i|<n} |Q(i\; T/n)|^{3/2}_{\mathcal{L}_2} \biggr)^2\\
&= \liminf_{n \to \infty} \frac{C_1}{\sqrt{T}\; (s_n^{(T)}\frac{T}{n})^{3/2}} \biggl( \sum_{|i|<n} |Q(i\; T/n)|^{3/2}_{\mathcal{L}_2}\; \frac{T}{n} \biggr)^2 \\
&=\frac{C_1}{\sqrt{T}\; s_T^{3/2}} \biggl( \int_{-T}^{T} |Q(t)|^{3/2}_{\mathcal{L}_2} dt \biggr)^2,
\end{aligned}
\end{equation}
and
\begin{equation}\label{eq: 4th cumul infty-D continuous}
\begin{aligned}
&\kappa_4(F_T) = \lim_{n \to \infty} \kappa_4(F_n^{(T)})  \leq  \liminf_{n \to \infty}  \frac{C_2}{n\; (s_n^{(T)})^2} \biggl( \sum_{|i|<n} |Q(i\; T/n)|^{4/3}_{\mathcal{L}_2} \biggr)^3\\
&\leq \liminf_{n \to \infty} \frac{C_2}{(s_n^{(T)}\frac{T}{n})^{2} \; T}\biggl(\sum_{|i|<n} |Q(i\; T/n)|^{4/3}_{\mathcal{L}_2}  \; \frac{T}{n}\biggr)^3 \\
&= \frac{C_2}{(s_T)^{2} \; T}\biggl(\int_{-T}^{T} |Q(t)|^{4/3}_{\mathcal{L}_2} dt \biggr)^3.
\end{aligned}
\end{equation}

\begin{Remark}
Real-valued Gaussian-subordinated processes are again a special case of the theory above.
\end{Remark}
\section{Berry-Esseen bounds for sample moments}\label{sec: Berry-Esseen bounds for sample moments}

In this section, we derive the upper bounds for the speed of convergence to normal distribution of the (centered and standardized) sample second moments calculated from the solutions to the original SPDE \eqref{eq: SPDE}. Our approach is motivated by the work \cite{SebaiyViens2016}, which covers the real-valued fractional Ornstein-Uhlenbeck processes. We study only the case $H < \frac{3}{4}$. Note that asymptotic normality in case  $H > \frac{3}{4}$ can in general not be expected (see discussion in subsection \ref {ex: parab. eq. distr noise}).\\

Recall that $(Z(t): t \geq 0)$ denotes the stationary solution to \eqref{eq: SPDE}, which is a stationary centered Gaussian process with values in $\mathcal{V}$ and with covariance operator $Q_\infty^\alpha \neq 0$ (by assumption \eqref{eq: Phi non zero}).\\

To obtain an explicit formula for the auto-covariance operators $Q_\infty^\alpha(t) = cov(Z(t),Z(0))$, we first formulate a proposition on calculating covariance of stochastic integrals driven by a fractional Brownian motion. This proposition is a straightforward generalization of Lemma 2.1 in \cite{Cheridito_et_al 2003}.

\begin{Proposition}\label{prop: Ito isometry fBm}
Let $(\beta^H(t):t \in \mathbb{R})$ be a scalar fractional Brownian motion with Hurst parameter $H \in (0,1)$. Consider real numbers $a<b \leq c < d$ and continuous deterministic functions $f \in \mathcal{C}([a,b])$, $g \in \mathcal{C}([c,d])$ such that
\[
\int_{a}^{b}f'(t)\beta^H(t) dt \quad \text{and} \quad \int_{c}^{d}g'(t)\beta^H(t) dt
\]
exist almost surely (as Riemann integrals). Then
\begin{equation}\label{eq: cov int fBm}
    \mathbb{E} \biggl(\int_{a}^{b}f(t) d\beta^H(t) \biggr) \biggl(\int_{c}^{d}g(s) d\beta^H(s) \biggr) = H(2H-1)\int_{a}^{b}\int_{c}^{d}f(t)g(s)(s-t)^{2H-2}ds dt.
\end{equation}
\end{Proposition}

Note that the equality \eqref{eq: cov int fBm} is well known in the regular case $H \geq \frac{1}{2}$. This proposition, however, covers also the singular cases $H < \frac{1}{2}$, because it relies on the disjoint integration domains.

\begin{proof}
The proof consists of integration by parts applied on the left-hand side, calculating expectations and applying reverse integration by parts. For details, see Lemma 2.1 in \cite{Cheridito_et_al 2003}.
\end{proof}

\begin{Lemma}\label{lem: auto-covariance operator}
Assume \eqref{eq: A0}, \eqref{eq: A1} and \eqref{eq: A2}. The auto-covariance operators of the stationary solution $(Z(t): t \geq 0)$ to \eqref{eq: SPDE} may be expressed by the formula
\begin{equation}\label{eq: auto-covariance operator}
\begin{aligned}
   &Q_\infty^\alpha(t) = cov(Z(t),Z(0)) \\
   &=Q_\infty^\alpha S_\alpha^{*}(t) + \int_{0}^{t} \int_{-\infty}^{0} S_\alpha(-r) \Phi \Phi^{*}S_\alpha^{*}(t-s)H(2H-1)(s-r)^{2H-2}drds.
\end{aligned}
\end{equation}
\end{Lemma}

\begin{proof}
Notice that the conditions \eqref{eq: A0} and \eqref{eq: A2} imply
\begin{equation}\label{analytic stability}
|S(r)\Phi|_{\mathcal{L}} \le C\frac{e^{-\rho r}}{r^{1-\epsilon}},\quad r>0,
\end{equation}
for a constant $C>0$, because we have
\[
|S(r)\Phi|_{\mathcal{L}} \le |S(r) (\beta I - A)^{1-\epsilon}|_{\mathcal{L}} \cdot |(\beta I - A)^{\epsilon-1}\Phi|_{\mathcal{L}(U, \mathcal{V})}
\]
and \eqref{analytic stability} follows by \cite{Pazy83}, Theorem 6.13. Consequently the second term on the right-hand side of (\ref{eq: auto-covariance operator}) is well defined and the integral converges in the operator norm. Indeed, for each $t\ge 0$ there is a constant $c_t$ such that
\begin{equation}\label{second term}
\begin{aligned}
  &\int_{0}^{t} \int_{-\infty}^{0} H(2H-1)(s-r)^{2H-2}|S_\alpha(-r) \Phi \Phi^{*}S_\alpha^{*}    (t-s)|_{\mathcal{L}}drds \\
  &\le c_t \int_{0}^{t} \int_{-\infty}^{0} e^{\rho r}e^{-\rho (t-s)} (-r)^{\epsilon -1} (t-s)^{\epsilon -1}(s-r)^{2H-2}drds <\infty
  \end{aligned}
\end{equation}
in view of Lemma \ref{lem: integral for Q} in the Appendix.
By construction of the stationary solution, for arbitrary $x,y \in Dom((\beta I -A^*)^{2-\epsilon})$ we have
\begin{equation*}
\begin{aligned}
   &\langle Q_\infty^\alpha(t) \; x,y \rangle_{\mathcal{V}} = \mathbb{E} \biggl( \langle Z(t),x \rangle_{\mathcal{V}}\;\langle Z(0),y \rangle_{\mathcal{V}} \biggr)  \\
   & = \mathbb{E}\biggl( \biggl\langle S_\alpha(t) Z(0) + \int_{0}^{t}S_\alpha(t-s)\Phi dB^{H}(s),x \biggr\rangle_{\mathcal{V}}\;\biggl\langle Z(0),y \biggr\rangle_{\mathcal{V}}\biggr)\\
   & = \mathbb{E}\biggl\langle S_\alpha(t) Z(0),x \biggr\rangle_{\mathcal{V}}\;\biggl\langle Z(0),y \biggr\rangle_{\mathcal{V}} \\
   &+ \mathbb{E}\biggl\langle  \int_{0}^{t}S_\alpha(t-s)\Phi dB^{H}(s),x \biggr\rangle_{\mathcal{V}}\;\biggl\langle \int_{-\infty}^{0}S_\alpha(-r)\Phi dB^{H}(r),y \biggr\rangle_{\mathcal{V}}\\
   &= (A) + (B).
\end{aligned}
\end{equation*}
For the first term, we have
\[
(A)=\mathbb{E}\biggl\langle Z(0),S_\alpha^{*}(t)x \biggr\rangle_{\mathcal{V}}\;\biggl\langle Z(0),y \biggr\rangle_{\mathcal{V}} = \langle Q_\infty^\alpha S_\alpha^{*}(t) \; x,y \rangle_{\mathcal{V}}.
\]
For the second term, we consider an orthonormal basis of $\mathcal{V}$ $\{e_k\}_{k=1}^{\infty}$  and denote by $\beta_{k}^{H}$, $k=1,2,...$ a system of independent $\mathbb{R}$-valued fractional Brownian motions, $\beta_{k}^{H}(t) = \langle B^{H}(t),e_k \rangle_{\mathcal{V}}$. Then
\begin{equation*}
\begin{aligned}
   &(B) = \mathbb{E}\biggl\langle  \sum_{k=1}^{\infty} \int_{0}^{t}S_\alpha(t-s)\Phi e_k d\beta_k^{H}(s),x \biggr\rangle_{\mathcal{V}}\;\biggl\langle \sum_{k=1}^{\infty} \int_{-\infty}^{0}S_\alpha(-r)\Phi e_k d\beta_k^{H}(r),y \biggr\rangle_{\mathcal{V}}\\
   &=\sum_{k=1}^{\infty}  \mathbb{E} \int_{0}^{t} \langle S_\alpha(t-s)\Phi e_k ,x \rangle_{\mathcal{V}} d\beta_k^{H}(s)\; \int_{-\infty}^{0} \langle S_\alpha(-r)\Phi e_k ,y \rangle_{\mathcal{V}}d\beta_k^{H}(r).
\end{aligned}
\end{equation*}

In view of the assumption \eqref{eq: A0} the functions $s\to \Phi ^*S^*_{\alpha}(t-s)x$ and $r\to \Phi ^*S^*_{\alpha}(-r)y$
are continuously differentiable on their respective domains, hence  we may apply Proposition \ref{prop: Ito isometry fBm} to obtain

\begin{equation*}
(B) = \sum_{k=1}^{\infty}  \int_{0}^{t} \int_{-\infty}^{0} \langle S_\alpha(t-s)\Phi e_k ,x \rangle_{\mathcal{V}} \; \langle S_\alpha(-r)\Phi e_k ,y \rangle_{\mathcal{V}} H(2H-1)(s-r)^{2H-2}drds.
\end{equation*}

Taking into account (\ref{second term}) and using the Dominated Convergence Theorem we have that

\begin{equation*}
\begin{aligned}
    &(B) = \int_{0}^{t} \int_{-\infty}^{0} \sum_{k=1}^{\infty}  \langle  e_k ,\Phi^{*} S_\alpha^{*}(t-s)x \rangle_{\mathcal{V}} \; \langle  e_k,\Phi^{*} S_\alpha^{*}(-r) y \rangle_{\mathcal{V}} H(2H-1)(s-r)^{2H-2}drds  \\
   &=\int_{0}^{t} \int_{-\infty}^{0} \langle  \Phi^{*} S_\alpha^{*}(t-s)x ,\Phi^{*} S_\alpha^{*}(-r) y \rangle_{\mathcal{V}} H(2H-1)(s-r)^{2H-2}drds  \\
   &=\int_{0}^{t} \int_{-\infty}^{0} \langle  S_\alpha(-r) \Phi \Phi^{*} S_\alpha^{*}(t-s)x , y \rangle_{\mathcal{V}} H(2H-1)(s-r)^{2H-2}drds  \\
   &= \biggl\langle \left( \int_{0}^{t} \int_{-\infty}^{0} S_\alpha(-r) \Phi \Phi^{*} S_\alpha^{*}(t-s) H(2H-1)(s-r)^{2H-2}drds \right) \; x , y \biggr\rangle_{\mathcal{V}}.
\end{aligned}
\end{equation*}
Since $x,y\in Dom((\beta I -A^*)^{2-\epsilon})$ are arbitrary and $Dom((\beta I -A^*)^{2-\epsilon})$ is dense in $\mathcal{V}$, this concludes the proof of (\ref{eq: auto-covariance operator}).
\end{proof}

\begin{Lemma}\label{lem: bound for Q}
The family of auto-covariance operators $(Q_\infty^\alpha(t): t \in \mathbb{R})$ is right-continuous with respect to the Hilbert-Schmidt norm and
\begin{equation}\label{eq: bound for Q}
    |Q_\infty^\alpha(t)|_{\mathcal{L}_2} \leq \min \{{\rm Tr} (Q_\infty^\alpha), C |t|^{2H-2}\}, \quad \forall t \in \mathbb{R},
\end{equation}
where $C$ is a constant which does not depend on $t$.
\end{Lemma}

\begin{proof}
Since $|Q_\infty^\alpha(-t)|_{\mathcal{L}_2}= |\left(Q_\infty^\alpha(t)\right)^{*}|_{\mathcal{L}_2} = |Q_\infty^\alpha(t)|_{\mathcal{L}_2} $, we can consider only the case $t\geq 0$.\\
By Cauchy-Schwarz inequality, we obtain
\[
 |Q_\infty^\alpha(t)|_{\mathcal{L}_2} \leq {\rm Tr} (Q_\infty^\alpha).
\]
For asymptotic behaviour, we can utilize the previous Lemma, which yields
\begin{equation*}
\begin{aligned}
&|Q_\infty^\alpha(t)|_{\mathcal{L}_2} \leq |Q_\infty^\alpha S_\alpha^{*}(t)|_{\mathcal{L}_2} + \int_{0}^{t} \int_{-\infty}^{0} |S_\alpha(-r) \Phi \Phi^{*}S_\alpha^{*}(t-s)|_{\mathcal{L}_2} \; H(2H-1)(s-r)^{2H-2}drds
\end{aligned}
\end{equation*}
Using the exponential stability \eqref{eq: A2}, the first term can easily be bounded by an exponential function:
\[
|Q_\infty^\alpha S_\alpha^{*}(t)|_{\mathcal{L}_2} = |S_\alpha(t) Q_\infty^\alpha|_{\mathcal{L}_2} \leq |S_\alpha(t)|_{\mathcal{L}} |Q_\infty^\alpha|_{\mathcal{L}_2} \leq M e^{-\rho \alpha t} |Q_\infty^\alpha|_{\mathcal{L}_2}.
\]

For the second term, we employ both assumptions \eqref{eq: A1} and \eqref{eq: A2} and the relation $S_\alpha(t) = S(\alpha t)$, which implies
\[
|S_\alpha(t) \Phi|_{\mathcal{L}_2} \leq c e^{-\rho \alpha t} t^{-\gamma}, \quad \forall t>0,
\]
for some constants $c>0$ and $\gamma \in [0,H)$. If we apply this inequality to the operators inside the integral, we obtain
\begin{equation*}
\begin{aligned}
&\int_{0}^{t} \int_{-\infty}^{0} |S_\alpha(-r) \Phi \Phi^{*}S_\alpha^{*}(t-s)|_{\mathcal{L}_2} \; H(2H-1)(s-r)^{2H-2}drds  \\
&\leq C \int_{0}^{t} \int_{-\infty}^{0} e^{-\rho \alpha (-r)} (-r)^{-\gamma} \; e^{-\rho \alpha (t-s)} (t-s)^{-\gamma} \; H(2H-1)(s-r)^{2H-2}drds \\
&\leq C t^{2H-2},
\end{aligned}
\end{equation*}
for some constant $C>0$ (see the Lemma \ref{lem: integral for Q} in the Appendix for details).
\end{proof}

For clarity of exposition, we shall describe construction of the Berry-Esseen type of bounds for discretely observed solution (the continuous-time case is analogous). The calculations below are similar to those presented in \cite{Nourdin Peccati 2012 book} in section 7.4, where increments of fBm are considered. The definitions of $F_n$ from \eqref{eq: def V_n F_n - infty} applied to the stationary solution to the SPDE take the form
\begin{equation*}
\begin{aligned}
&F_n = \frac{1}{\sqrt{s_n}}\frac{1}{\sqrt{n}} \sum_{i=1}^{n}(|Z(i)|_\mathcal{V}^2 - {\rm Tr}(Q_{\infty}^{\alpha})),\\
\end{aligned}
\end{equation*}
where
\begin{equation}\label{eq: def sn for solution}
\begin{aligned}
&s_n = 2\sum_{i=-(n-1)}^{n-1}\biggl(1 - \frac{|i|}{n} \biggr) |Q_\infty^\alpha(i)|^2_{\mathcal{L}_2}.\\
\end{aligned}
\end{equation}

The upper bounds of auto-covariance operators \eqref{eq: bound for Q} yield
\[
\sum_{|i|<n} |Q_\infty^\alpha(i)|^{3/2}_{\mathcal{L}_2} \leq
	\begin{cases}
		C , \quad H< \frac{2}{3}, \\
        C \ln n, \quad H=\frac{2}{3}, \\
        C n^{3H-2}, \quad \frac{2}{3} <H, \\
	\end{cases}
\]
and
\[
\sum_{|i|<n} |Q_\infty^\alpha(i)|^{4/3}_{\mathcal{L}_2} \leq
	\begin{cases}
		C , \quad H< \frac{5}{8}, \\
        C \ln n, \quad H=\frac{5}{8}, \\
        C n^{(8H-5)/3}, \quad \frac{5}{8} <H, \\
	\end{cases}
\]
where $C$ is a (sufficiently large) constant.\\

If we plug these calculations into the bounds on cumulants \eqref{eq: 3  cumul bound infty D}, \eqref{eq: 4  cumul bound infty D}, we obtain
\begin{equation}\label{eq: B-E bounds for max cumulants - inftyD}
\max \{|\kappa_3(F_n)|,|\kappa_4(F_n)|\}
 \leq
	\begin{cases}
		C \frac{1}{\sqrt{n}}, \quad H< \frac{2}{3}, \\
        C \frac{(\ln n)^{2}}{\sqrt{n}}, \quad H=\frac{2}{3}, \\
        C n^{6H-9/2}, \quad \frac{2}{3} <H<\frac{3}{4}. \\
	\end{cases}
\end{equation}

Combining this with the $4^{th}$ moment theorem (the formula \eqref{eq:optimum B-E bounds}) we obtain

\begin{equation}\label{eq: B-E bounds for F_n - inftyD}
d_{TV}(F_n,N) \leq
	\begin{cases}
		C \frac{1}{\sqrt{n}}, \quad H< \frac{2}{3}, \\
        C \frac{(\ln n)^{2}}{\sqrt{n}}, \quad H=\frac{2}{3}, \\
        C n^{6H-9/2}, \quad \frac{2}{3} <H<\frac{3}{4}, \\
	\end{cases}
\end{equation}
where $C$ stands for a constant independent of $n$.\\

Next, we eliminate the effect of standardizing factors $s_n$ from $F_n$. Observe that for $H < \frac {3}{4}$:
\begin{equation}\label{eq: def v inftyD}
    \lim_{n\to \infty} s_n = 2\sum_{i=-\infty}^{\infty}|Q_\infty^\alpha(i)|^2_{\mathcal{L}_2} =: s_\infty^{*} < \infty.
\end{equation}

Replacing the standardizing factors by their limits leads to the following Berry-Esseen-type of bounds:

\begin{Lemma}\label{lem: B-E bounds w/o s_n - inftyD}
Consider $F_n, s_n, s_\infty^{*}$ defined above and $N \sim \mathcal{N}(0,1)$. Then
\begin{equation}\label{eq: B-E bounds w/o s_n  - inftyD}
    d_{TV}\biggl( \frac{\sqrt{s_n}}{\sqrt{s_\infty^{*}}} F_n, N\biggr) \leq 2 |1-\frac{s_n}{s_\infty^{*}}| + C \max \{|\kappa_3(F_n)|,|\kappa_4(F_n)|\}.
\end{equation}
\end{Lemma}
\begin{proof}
Denote by $D$ and $L^{-1}$ the Malliavin derivative and the pseudo-inverse of the generator of the Ornstein-Uhlenbeck semigroup with respect to the isonormal Gaussian process $\mathbb{X}$ on Hilbert space $\mathcal{H}$ generated from the stationary solution $Z$ as described in Section \ref{subsec: infty D sequences}. In particular, for random elements from the second Wiener chaos of $\mathbb{X}$, we have $L^{-1} I_2(f) = -\frac{1}{2}I_2(f)$. For more details on these two linear operators, consult e.g. \cite{Nourdin Peccati 2012 book}. It immediately follows from the proof of Theorem 1.2 in \cite{Nourdin Peccati 2015} that
\begin{equation*}
\begin{aligned}
   &d_{TV}\biggl( \frac{\sqrt{s_n}}{\sqrt{s_\infty^{*}}} F_n, N\biggr) \leq 2 \mathbb{E}\left[\left| \mathbb{E}(1-\langle D\frac{\sqrt{s_n}}{\sqrt{s_\infty^{*}}} F_n,-DL^{-1}\frac{\sqrt{s_n}}{\sqrt{s_\infty^{*}}} F_n \rangle_{\mathcal{H}}| F_n) \right| \right] \\
   &= 2 \mathbb{E}\left[\left| \mathbb{E}(1- \frac{s_n}{s_\infty^{*}} +\frac{s_n}{s_\infty^{*}} - \frac{s_n}{s_\infty^{*}} \langle D F_n,-DL^{-1} F_n \rangle_{\mathcal{H}}| F_n) \right| \right] \\
   &\leq 2 |1-\frac{s_n}{s_\infty^{*}}| + 2 \frac{s_n}{s_\infty^{*}} \mathbb{E}\left[\left| \mathbb{E}(1- \langle D F_n,-DL^{-1} F_n \rangle_{\mathcal{H}}| F_n) \right| \right] .
\end{aligned}
\end{equation*}
Clearly $ 0 < \frac{s_n}{s_\infty^{*}} < 1$ and the term $\mathbb{E}\left[\left| \mathbb{E}(1- \langle D F_n,-DL^{-1} F_n \rangle_{\mathcal{H}}| F_n) \right| \right]$ can be bounded above by $C \max \{|\kappa_3(F_n)|,|\kappa_4(F_n)|\}$ (which again follows from the proof of Theorem 1.2 in \cite{Nourdin Peccati 2015}).
\end{proof}

Note that by Remark \ref{rem: 4th moment for d_W } we can reformulate previous results in terms of the Wasserstein distance. \\

Next step is to extend the result from the stationary solutions to the general solutions to \eqref{eq: SPDE}. Following the approach in \cite{SebaiyViens2016} we consider the Wasserstein distance. \\

Consider a solution $(X(t),t \geq 0)$ to \eqref{eq: SPDE} with a general initial condition $X_0$ and the stationary solution $(Z(t),t \geq 0)$. Clearly
\[
X(t) = Z(t) + S_\alpha(t)(X(0)-Z(0)).
\]
Furthermore, the effect of the non-stationary term can then be controlled by the following elementary (but useful) lemma (\cite{SebaiyViens2016}):

\begin{Lemma}\label{lem: Wasserstein distance for sum}
Consider two random variables $X$ and $Z$ defined on the same probability space. Then
\[
d_W(X+Z,N) \leq d_W(Z,N) + \mathbb{E}|X|.
\]
\end{Lemma}

\begin{Theorem}\label{thm: B-E bounds for moments - inftyD}
Let \eqref{eq: A0}, \eqref{eq: A1} and \eqref{eq: A2} be satisfied and let the previous notation prevail. Let $N_v$ denote a Gaussian random variable,
\[
N_v \sim \mathcal{N}(0,s_\infty^{*}),
\]
and define the upper-bound function:
\begin{equation}\label{eq: def xi}
 \xi_H(n) := \begin{cases}
		\frac{1}{\sqrt{n}}, \quad H \leq \frac{5}{8}, \\
        \frac{1}{n^{3-4H}}, \quad \frac{5}{8} <H<\frac{3}{4}.
	\end{cases}
\end{equation}
Then for the stationary solution $(Z(t),t \geq 0)$ to the equation \eqref{eq: SPDE} we have that
\begin{equation}\label{eq: B-E bounds for stat moments - inftyD}
\begin{aligned}
& d_{TV}\biggl[\sqrt{n}\left(\biggl(\frac{1}{n}\sum_{i=1}^{n}|Z(i)|_\mathcal{V}^2\biggr) - \frac{1}{\alpha^{2H}} {\rm Tr}(Q_{\infty})\right),N_v \biggr] \leq C \; \xi_H(n).\\
\end{aligned}
\end{equation}
Now consider a general solution $(X(t),t \geq 0)$ to \eqref{eq: SPDE} with initial condition $X_0$ such that $\mathbb{E}|X_0|^2_\mathcal{V} < \infty$. Then
\begin{equation}\label{eq: B-E bounds for non-stat moments - inftyD}
\begin{aligned}
& d_{W}\biggl[\sqrt{n}\left(\biggl(\frac{1}{n}\sum_{i=1}^{n}|X(i)|_\mathcal{V}^2\biggr) - \frac{1}{\alpha^{2H}} {\rm Tr}(Q_{\infty})\right),N_v \biggr] \leq C \; \xi_H(n).\\
\end{aligned}
\end{equation}
\end{Theorem}

\begin{proof}
 In the stationary case, we use Lemma \ref{lem: B-E bounds w/o s_n - inftyD} to see
\begin{equation*}
\begin{aligned}
 &d_{TV}\biggl[\sqrt{n}\left(\biggl(\frac{1}{n}\sum_{i=1}^{n}|Z(i)|_\mathcal{V}^2\biggr) - \frac{1}{\alpha^{2H}} {\rm Tr}(Q_{\infty})\right),N_v \biggr] \\
 &\leq 2 |1-\frac{s_n}{s_\infty^{*}}| + C \max \{|\kappa_3(F_n)|,|\kappa_4(F_n)|\}.
\end{aligned}
\end{equation*}
The term $\max \{|\kappa_3(F_n)|,|\kappa_4(F_n)|\}$ can be estimated as in \eqref{eq: B-E bounds for max cumulants - inftyD}.\\

Next, we use the formulas \eqref{eq: def sn for solution} and \eqref{eq: def v inftyD} together with \eqref{eq: bound for Q} to obtain
\[
|1-\frac{s_n}{s_\infty^{*}}| \leq
	\begin{cases}
		C \frac{1}{n}, \quad H< \frac{1}{2}, \\
        C \frac{\ln n}{n}, \quad H=\frac{1}{2}, \\
        C n^{4H-3}, \quad \frac{1}{2} < H <\frac{3}{4}. \\
	\end{cases}
\]
Combining these two bounds, we arrive at \eqref{eq: B-E bounds for stat moments - inftyD}.\\

For the non-stationary case, we apply Lemma \ref{lem: Wasserstein distance for sum}. Consequently, we need to calculate
\begin{equation*}
\begin{aligned}
& \mathbb{E} \biggl| |X(i)|_\mathcal{V}^2 - |Z(i)|_\mathcal{V}^2\biggr| \leq \sqrt{\mathbb{E}|X(i) - Z(i)|^2_\mathcal{V}}  \; \biggl(\sqrt{\mathbb{E}(|X(i)|^2_\mathcal{V}} + \sqrt{\mathbb{E}(|Z(i)|^2_\mathcal{V}} \biggr).
\end{aligned}
\end{equation*}
For the first factor on the r.h.s., we have
\begin{equation*}
\begin{aligned}
&\sqrt{\mathbb{E}|X(i) - Z(i)|^2_\mathcal{V}} =  \sqrt{\mathbb{E}|S_\alpha(i)(X(0) - Z(0))|^2_\mathcal{V}} \\
&\leq |S_\alpha(i)|_{\mathcal{L(V)}}\sqrt{\mathbb{E}|(X(0) - Z(0))|^2_\mathcal{V}} \leq C \; e^{-\rho i},
\end{aligned}
\end{equation*}
where we employed the assumption $\mathbb{E}|X_0|^2_\mathcal{V} < \infty$ and the exponential stability \eqref{eq: A2}.\\
The second factor can be bounded above by a constant, since (using the above bound)
\begin{equation*}
\begin{aligned}
&\sqrt{\mathbb{E}(|X(i)|^2_\mathcal{V}} + \sqrt{\mathbb{E}(|Z(i)|^2_\mathcal{V}} \leq \sqrt{\mathbb{E}(|X(i) - Z(i)|^2_\mathcal{V}} + \sqrt{\mathbb{E}(|Z(i)|^2_\mathcal{V}} + \sqrt{\mathbb{E}(|Z(i)|^2_\mathcal{V}}\\
&\leq  C \; e^{-\rho i} + 2 {\rm Tr}(Q_{\infty}^{\alpha}).
\end{aligned}
\end{equation*}
Hence,
\[
\mathbb{E} \biggl||X(i)|^2_\mathcal{V} - |Z(i)|^2_\mathcal{V}\biggr|\leq C e^{-\rho i},
\]
where $C$ is a constant independent of $i$ and $\rho$ is the coefficient in the exponential stability \eqref{eq: A2}. It follows directly
\[
\mathbb{E}\biggl|  \frac{1}{\sqrt{n}}\sum_{i=1}^{n}\biggl(|X(i)|^2_\mathcal{V} - |Z(i)|^2_\mathcal{V}\biggr)\biggr| \leq C \frac{1}{\sqrt{n}},
\]
where $C$ does not depend on $n$. This, according to Lemma \ref{lem: Wasserstein distance for sum}, does not distort the upper bound $\xi_H(n)$.\\

The continuous-time case can by treated similarly.
\end{proof}

\begin{Remark}\label{rem: BE for moments - ifntyD, continuous}
In continuous-time case, we can proceed analogously to show
\begin{equation*}
\begin{aligned}
&d_{TV}\biggl[\sqrt{T}\left(\biggl(\frac{1}{T}\int_{0}^{T}|Z(t)|_\mathcal{V}^2 dt \biggr) - \frac{1}{\alpha^{2H}} {\rm Tr}(Q_{\infty})\right), N_u \biggr] \leq C \; \xi_H(T),\\
& d_{W}\biggl[\sqrt{T}\left(\biggl(\frac{1}{T}\int_{0}^{T}|X(t)|_\mathcal{V}^2 dt \biggr) - \frac{1}{\alpha^{2H}} {\rm Tr}(Q_{\infty})\right), N_u \biggr] \leq C \; \xi_H(T),\\
\end{aligned}
\end{equation*}
where $N_u \sim \mathcal{N}(0,u_\infty^{*})$ and $u_\infty^{*} = 2\int_{-\infty}^{\infty} |Q_\infty^\alpha(t)|^2_{\mathcal{L}_2}(t) dt < \infty$.
\end{Remark}

\begin{Remark}\label{rem: BE for moments - 1D}
If we observe a one-dimensional projection of the solution, we can proceed similarly. Considering the stationary centered Gaussian process  $z_t := \langle Z(t),w \rangle_\mathcal{V}$ with auto-covariance function
\begin{equation}\label{eq: def r_z}
r_z^{\alpha}(t) = \mathbb{E}(z_t z_0) = \langle Q_{\infty}^{\alpha}(t) w,w \rangle_\mathcal{V},
\end{equation}
we can show again that $r_z^{\alpha}$ is right-continuous with
\begin{equation}\label{eq: bound for r_z}
    |r_z^{\alpha}(t)| \leq \min \{r_z^{\alpha}(0), C |t|^{2H-2}\}, \quad \forall t \in \mathbb{R}.
\end{equation}
The above approach then leads to the following Berry-Esseen bounds for the stationary solution:
\begin{equation}\label{eq: B-E bounds for stat moments - 1D}
\begin{aligned}
& d_{TV}\biggl[\sqrt{n}  \left( \biggl(\frac{1}{n}\sum_{i=1}^{n}\langle Z(i), w\rangle^2_\mathcal{V}\biggr) - \frac{1}{\alpha^{2H}} \langle Q_{\infty} w, w\rangle_\mathcal{V}\right),N_v \biggr] \leq C \; \xi_H(n),\\
& d_{TV}\biggl[\sqrt{T} \left( \biggl( \frac{1}{T}\int_{0}^{T}\langle Z(t), w\rangle^2_\mathcal{V} dt \biggr) - \frac{1}{\alpha^{2H}} \langle Q_{\infty} w, w\rangle_\mathcal{V} \right), N_u \biggr] \leq C \; \xi_H(T),\\
\end{aligned}
\end{equation}
and for a general solution $(X(t),t \geq 0)$ to \eqref{eq: SPDE} with initial condition $X_0$ such that $\mathbb{E}|X_0|^2_\mathcal{V} < \infty$:
\begin{equation}\label{eq: B-E bounds for non-stat moments - 1D}
\begin{aligned}
& d_{W}\biggl[\sqrt{n}\left(\biggl(\frac{1}{n}\sum_{i=1}^{n}\langle X(i), w\rangle^2_\mathcal{V}\biggr) - \frac{1}{\alpha^{2H}}\langle Q_{\infty} w, w\rangle_\mathcal{V}\right),N_v \biggr] \leq C \; \xi_H(n),\\
& d_{W}\biggl[\sqrt{T}\left(\biggl(\frac{1}{T}\int_{0}^{T}\langle X(t), w\rangle^2_\mathcal{V} dt \biggr) - \frac{1}{\alpha^{2H}}\langle Q_{\infty} w, w\rangle_\mathcal{V}\right) , N_u \biggr] \leq C \; \xi_H(T),\\
\end{aligned}
\end{equation}
where $\xi_H(n)$ is defined in Theorem \ref{thm: B-E bounds for moments - inftyD},  $N_v \sim \mathcal{N}(0,s_\infty)$, $N_u \sim \mathcal{N}(0,u_\infty)$), $s_\infty =2\sum_{i=-\infty}^{\infty}(r_z^{\alpha}(i))^2$ and $ u_\infty = 2\int_{-\infty}^{\infty} (r_z^{\alpha}(t))^2 dt$.
\end{Remark}

\section{Berry-Esseen bounds for estimators}\label{sec: Berry-Esseen bounds for estimators}

Recall that the estimators \eqref{eq: MC_est_continuous}, \eqref{eq: MC_est_discrete}, \eqref{eq: proj_est_continuous} and \eqref{eq: proj_est_discrete} are all constructed (under assumption \eqref{eq: Phi non zero}) as monotonous, twice differentiable functions of the corresponding sample second moments, whose asymptotic properties were studied in the previous section. It is well-known, that this monotonous differentiable transformation does not distort asymptotic normality. However, when constructing the Berry-Esseen bounds for the transformed processes, we have to switch to the Kolmogorov distance localized on compacts, as suggested by the following proposition:

\begin{Proposition} \label{prop:BE bounds for transformed}
Denote by $\Psi$ the distribution function of the standard normal distribution $\mathcal{N}(0,1)$ and consider a stochastic process $(U_T: T >0)$ with mean value $\mu$ and a standardizing function $\sigma_T$, $\sigma_T \overset{T \to \infty}{\longrightarrow}  0$, such that
\[
\sup_{z\in \mathbb{R}} \biggl| \mathbb{P}\biggl(\frac{U_T-\mu}{\sigma_{T}} \leq z\biggr) - \Psi(z) \biggr| \leq \xi(T), \quad \forall T>0,
\]
$\xi(T)$ being the upper bound for the Kolmogorov distance. \\

Now let $g$ be a monotonous function, $g \in  \mathcal{C}^{2}(A)$, where $\mathbb{P}(U_T\in A) = 1$ for all $T$. Then for each $K>0$ there exists a constant $C_{K}$ such that
\[
\sup_{z\in [-K,K]} \biggl|\mathbb{P}\biggl(\frac{g(U_T)-g(\mu)}{|g'(\mu)| \sigma_{T}} \leq z\biggr) - \Psi(z) \biggr| \leq  C_{K} \max \{\xi(T),\sigma_{T}\} , \quad \forall T>0.
\]
\end{Proposition}

\begin{proof}
The idea of the proof follows calculations from \cite{Sottinen-Vitasaary} (see the proof of Theorem 3.2. therein). Denote $\psi := g^{-1}$, $\phi := g(\mu)$ and assume $g^{\prime} <0$ (the case $g^{\prime} >0$ is similar). It is easy to see that

\begin{equation} \label{eq: BE transformed 1}
\begin{aligned}
&\biggl|\mathbb{P}\biggl(\frac{g(U_T)-g(\mu)}{|g'(\mu)| \sigma_{T}} \leq z\biggr) - \Psi(z) \biggr|  \\
&\leq \biggl|\mathbb{P}\biggl(\frac{U_T-\mu}{\sigma_{T}} \geq \nu \biggr) - (1-{\Psi}(\nu)) \biggr| + \biggl| \Psi(-\nu) - \Psi(z) \biggr| \leq \xi(T) + |-\nu - z|,
\end{aligned}
\end{equation}
where
\[
\nu =\frac{\psi(z\sigma_{T}|g'(\mu)| + \phi) - \psi({\phi})}{\sigma_{T}} = \frac{1}{\sigma_{T}} \psi'(\eta_{T})z\sigma_{T}|g'(\mu)|=z\frac{\psi'(\eta_{T})}{|\psi'(\phi)|}=-z\frac{\psi'(\eta_{T})}{\psi'(\phi)},
\]
with $\eta_{T} \in (\phi, \phi + z\sigma_{T}|g'(\mu)|)$. Now

\begin{equation} \label{eq: BE transformed 2}
\begin{aligned}
&|-\nu - z| = |z\frac{\psi'(\eta_{T})}{\psi'(\phi)} - z|=|\frac{z}{\psi'(\phi)}|\;|\psi'(\eta_{T})-\psi'(\phi)| \\
&= |\frac{z}{\psi'(\phi)}|\;|\psi''(\xi_{T})(\eta_{T}-\phi)|\leq |\frac{z}{\psi'(\phi)}|\;|\psi''(\xi_{T})| \; |z\sigma_{T}|\;|g'(\mu)| \\
&\leq |z|^{2}|g'(\mu)|^{2}\;\sigma_{T}\; \sup \{ |\psi''(y)|\; : \; y \in [g(\mu) - z\sigma_{T}|g'(\mu)|,\;g(\mu) + z\sigma_{T}|g'(\mu)|]\}.
\end{aligned}
\end{equation}
Clearly for any $K>0$ there is $\delta >0$ and $T_{0}>0$ such that
\[
\sup_{|z|\leq K}|-\nu - z| \leq K^{2}|g'(\mu)|^{2}\;\sigma_{T}\;\sup \{|\psi''(y)|\; : \; y \in [g(\mu) - \delta,\;g(\mu) + \delta]\}  \leq C_{K} \sigma_{T},
\]
for all $T>T_{0}$.
Combining this estimate with \eqref{eq: BE transformed 1} we obtain the desired result.
\end{proof}

We shall first formulate the Berry-Esseen bounds for the estimators constructed by means of the stationary solution. The minimum contrast estimators are well suited for these situations. \\

Note that the Berry-Esseen bounds for the sample moments from previous section can be reformulated in terms of the Kolmogorov distance, therefore Proposition \ref{prop:BE bounds for transformed} can be used.

\begin{Theorem}\label{thm: B-E bounds for estimators - stationary}
Let \eqref{eq: A0}, \eqref{eq: A1} and \eqref{eq: A2} be satisfied. Consider the estimators $\hat{\alpha}_{T}, \check{\alpha}_{n}, \tilde{\alpha}_{T}, \bar{\alpha}_{n}$ defined in \eqref{eq: MC_est_continuous} - \eqref{eq: proj_est_discrete} based on the \textbf{stationary solution} to the equation \eqref{eq: SPDE} (i.e. $X(t) = Z(t)$) and recall the upper-bound function:
\begin{equation}\label{eq: def xi_estim_stat}
 \xi_H(t) := \begin{cases}
		\frac{1}{\sqrt{t}}, \quad H \leq \frac{5}{8}, \\
        \frac{1}{t^{3-4H}}, \quad \frac{5}{8} <H<\frac{3}{4}.
	\end{cases}
\end{equation}
Then for each $K>0$, there exists a constant $C_K$, such that
\begin{equation*}
\begin{aligned}
&\sup_{z\in [-K,K]} \biggl|\mathbb{P}\biggl(\sqrt{n}\frac{\check{\alpha}_{n} - \alpha}{\gamma_{\alpha} \; \sqrt{s_\infty^{*}}} \leq z\biggr) - \Psi(z) \biggr| \leq  C_{K} \; \xi_H(n), \quad \forall n>0,\\
&\sup_{z\in [-K,K]} \biggl|\mathbb{P}\biggl(\sqrt{T}\frac{\hat{\alpha}_{T} - \alpha}{\gamma_{\alpha} \; \sqrt{u_\infty^{*}}} \leq z\biggr) - \Psi(z) \biggr| \leq  C_{K} \; \xi_H(T), \quad \forall T>T_0 \quad \text{for some }T_0>0,\\
&\sup_{z\in [-K,K]} \biggl|\mathbb{P}\biggl(\sqrt{n}\frac{\bar{\alpha}_{n} - \alpha}{\delta_{\alpha} \; \sqrt{s_\infty}} \leq z\biggr) - \Psi(z) \biggr| \leq  C_{K} \; \xi_H(n), \quad \forall n>0,\\
&\sup_{z\in [-K,K]} \biggl|\mathbb{P}\biggl(\sqrt{T}\frac{\tilde{\alpha}_{T} - \alpha}{\delta_{\alpha} \; \sqrt{u_\infty}} \leq z\biggr) - \Psi(z) \biggr| \leq  C_{K} \; \xi_H(T), \quad \forall T>T_0 \quad \text{for some }T_0>0,\\
\end{aligned}
\end{equation*}
where
\begin{equation*}
\begin{aligned}
&\gamma_{\alpha} = \frac{\alpha^{1+2H}}{2H \; {\rm Tr} (Q_{\infty})},\\
&\delta_{\alpha} = \frac{\alpha^{1+2H}}{2H \; \langle Q_{\infty} w,w\rangle_\mathcal{V}}.
\end{aligned}
\end{equation*}
\end{Theorem}

\begin{proof}
It is a direct consequence of the Theorem \ref{thm: B-E bounds for moments - inftyD} (cf. also Remarks \ref{rem: BE for moments - ifntyD, continuous} and \ref{rem: BE for moments - 1D}) and the Proposition \ref{prop:BE bounds for transformed}. \\
\end{proof}

\begin{Remark}\label{rem: B-E bounds for estimators - non-stationary}
In case of \textbf{general (non-stationary) solution}, we can proceed similarly. First, we have to replace the Wasserstein distance by the Kolmogorov distance. For this purpose, we can utilize the following well-known general inequality:
\begin{equation}\label{eq: Kol vs Was}
d_{Kol}(X,N) \leq 2\sqrt{C\; d_{W}(X,N)},
\end{equation}
where $X$ and $N$ are any random variables and $N$ has absolutely continuous distribution with the density function bounded by the constant $C$. This inequality, however, decelerates the speed of convergence of the local Berry-Esseen bounds. \\
In view of Theorem \ref{thm: B-E bounds for moments - inftyD} (cf. Remarks \ref{rem: BE for moments - ifntyD, continuous} and \ref{rem: BE for moments - 1D}) and Proposition \ref{prop:BE bounds for transformed} we obtain results corresponding to Theorem \ref{thm: B-E bounds for estimators - stationary}, but with upper bounds for the local Kolmogorov distance being $c_{K} \; \sqrt{\xi_H(n)}$ and $c_{K} \; \sqrt{\xi_H(T)}$ instead of $C_{K} \; \xi_H(n)$ and $C_{K} \; \xi_H(T)$, respectively.
\end{Remark}

The central limit theorem for the estimators is now an easy corollary to the previous theorem.
\begin{Corollary}\label{cor: CLT}
Assume \eqref{eq: A0}, \eqref{eq: A1}, \eqref{eq: A2}, $\mathbb{E}|X_0|^2_\mathcal{V} < \infty$ and $H<\frac{3}{4}$. For the estimators \eqref{eq: MC_est_continuous}, \eqref{eq: MC_est_discrete}, \eqref{eq: proj_est_continuous} and \eqref{eq: proj_est_discrete} the central limit theorems
\begin{equation*}
\begin{aligned}
&\sqrt{n}(\check{\alpha}_{n} - \alpha) \overset{d}{\to} \mathcal{N}(0,\sigma_1^{2}),\\
&\sqrt{T}(\hat{\alpha}_{T} - \alpha) \overset{d}{\to} \mathcal{N}(0,\sigma_2^{2}),\\
&\sqrt{n}(\bar{\alpha}_{n} - \alpha) \overset{d}{\to} \mathcal{N}(0,\sigma_3^{2}),\\
&\sqrt{T}(\tilde{\alpha}_{T} - \alpha) \overset{d}{\to} \mathcal{N}(0,\sigma_4^{2}),\\
\end{aligned}
\end{equation*}
hold true for
\begin{equation*}
\begin{aligned}
&\sigma_1 = \frac{\alpha^{1+2H}}{2H \; {\rm Tr} (Q_{\infty})} \sqrt{2\sum_{i=-\infty}^{\infty}|Q_\infty^\alpha(i)|^2_{\mathcal{L}_2}} ,\\
&\sigma_2 = \frac{\alpha^{1+2H}}{2H \; {\rm Tr} (Q_{\infty})} \sqrt{2\int_{-\infty}^{\infty} |Q_\infty^\alpha(t)|^2_{\mathcal{L}_2}(t) dt} ,\\
&\sigma_3 = \frac{\alpha^{1+2H}}{2H \; \langle Q_{\infty} w,w\rangle_\mathcal{V}} \sqrt{2\sum_{i=-\infty}^{\infty}(\langle Q_{\infty}^{\alpha}(i) w,w \rangle_\mathcal{V})^2}, \\
&\sigma_4 = \frac{\alpha^{1+2H}}{2H \; \langle Q_{\infty} w,w\rangle_\mathcal{V}} \sqrt{2\int_{-\infty}^{\infty} (\langle Q_{\infty}^{\alpha}(t) w,w \rangle_\mathcal{V})^2 dt}.\\
\end{aligned}
\end{equation*}
\end{Corollary}

\begin{proof}
Observe that for $H<\frac{3}{4}$ we have $\xi_H(n) \to 0$ with $n \to \infty$.
\end{proof}

\section{Examples}\label{sec: examples}

\subsection{Parabolic equation with distributed fractional noise}\label{ex: parab. eq. distr noise}
Consider the following formal parabolic equation with fractional noise

\begin{equation}\label{eq: parabolic_SPDE}
\frac{\partial u}{\partial t}(t,\xi) = - \alpha (-\Delta)^{m} u (t,\xi) + \eta^H(t,\xi), \quad \text{ for } (t,\xi) \in \mathbb{R}_+ \times \mathcal{O},
\end{equation}
accompanied by the initial condition
\begin{equation}\label{eq: parabolic_SPDE_IC}
u(0,\xi) = x(\xi), \quad \xi \in \mathcal{O},
\end{equation}
and the Dirichlet boundary conditions
\begin{equation}\label{eq: parabolic_SPDE_BC}
\forall j = 0,...,m-1: \quad \frac{\partial^j u}{\partial \nu^j}(t,\xi) = 0, \quad \text{ for } (t,\xi) \in [0,+\infty) \times \partial \mathcal{O},
\end{equation}

where $\mathcal{O}$ is a bounded domain in $\mathbb{R}^d$ with smooth boundary $\partial \mathcal{O}$, $\Delta$ the Laplace operator, $\alpha > 0$ an unknown parameter and $\frac{\partial}{\partial \nu}$ denotes the normal derivative. The noise term $\eta^H(t,\xi)$ can be viewed as a formal time-derivative of a fractional Brownian motion with Hurst parameter $H \in (0,1)$, either white or correlated in space. Note that the particular case $m=1$ corresponds to the stochastic heat equation.\\

We can reformulate the above parabolic problem rigorously as the stochastic evolution equation
\begin{equation} \label{eq: parabolic SEE}
\begin{aligned}
dX(t) &= \alpha A X(t) dt + \Phi dB^{H}(t),\\
X(0) &= X_0,
\end{aligned}
\end{equation}
in the Hilbert space $\mathcal{V} = U = L^2(\mathcal{O})$, where $X_0 \in L^2(\mathcal{O})$, $A = -(-\Delta)^{m}$ with domain
\[
Dom(A) = \{y \in W^{2m}_{2}(\mathcal{O}): \frac{\partial^j y}{\partial \nu^j} = 0 \text{ on } \partial \mathcal{O}, \; j=0,...,m-1 \},
\]
where $W^{2m}_{2}(\mathcal{O})$ is the standard Sobolev space (see e.g. \cite{Triebel83}). For the noise part, $B^H$ denotes the cylindrical fractional Brownian motion with Hurst parameter $H$ and $\Phi$, being the bounded operator on $L^2(\mathcal{O})$, determines the space correlations. It clearly fulfills the condition \eqref{eq: A0} with $\epsilon = 1$.\\

It is a well-known fact that $A$ generates an analytic semigroup $(S(t):t\geq 0)$, for which
\[
|S(t)\Phi|_{\mathcal{L}_2} \leq |S(t)|_{\mathcal{L}_2}|\Phi|_{\mathcal{L}} \leq ct^{-\frac{d}{4m}}, \quad \forall t \in (0,T],
\]
and
\[
|S(t)|_{\mathcal{L}}\leq M e^{-\rho_1 t}, \quad \forall t > 0,
\]
where $-\rho_1 < 0$ is the greatest eigenvalue of $A$.
Hence, conditions \eqref{eq: A1} and \eqref{eq: A2} are fulfilled whenever $H > \frac{d}{4m}$. In this case, we can apply Theorem \ref{cor: consistency} to see that the estimators \eqref{eq: MC_est_continuous} and \eqref{eq: MC_est_discrete} (defined if $\Phi \neq 0$) and estimators \eqref{eq: proj_est_continuous} and \eqref{eq: proj_est_discrete} (defined if $\langle Q_{\infty} w,w\rangle_\mathcal{V} \neq 0$) are strongly consistent. If we assume, in addition, that $H<\frac{3}{4}$, Corollary \ref{cor: CLT} implies asymptotic normality of the estimators with the upper bounds for the speed of the convergence given in Theorem \ref{thm: B-E bounds for estimators - stationary} (resp. Remark \ref{rem: B-E bounds for estimators - non-stationary}).\\

\begin{Remark}\label{rem: non-CLT}
Note that for $H > \frac{3}{4}$, Corollary \ref{cor: CLT} may not hold. If the equation \eqref{eq: parabolic SEE} is diagonalizable (e.g. if $\Phi$ is identity so that the noise is white in space), the projection $(\langle X(t),e_k \rangle_\mathcal{V}, t \geq 0)$ onto the common eigenvector $e_k$ of the operators $\Phi$ and $A$ is a one-dimensional fractional Ornstein-Uhlenbeck process. It is well known for such process (see e.g. Remark 16 in \cite{SebaiyViens2016})
that if $H > \frac{3}{4}$, the sequence
\[
\frac{\sum_{i=1}^{n}|\langle X(i),e_k \rangle_\mathcal{V}|^{2}-\langle Q_{\infty} e_k,e_k\rangle_\mathcal{V}}{n^{2H-1}}
\]
converges in distribution to the (non-Gaussian) law of the second-chaos variable, also called the (scaled) Rosenblatt law.
\end{Remark}

\subsection{Parabolic equation with pointwise noise} \label{ex: parab. eq. point noise}
Consider the heat equation
\begin{equation} \label{eq: heat eq point noise - naive}
\frac{\partial u}{\partial t}(t,\xi) =  \alpha \Delta u(t,\xi) + \delta_y \eta^H(t), \quad \text{ for } (t,\xi) \in \mathbb{R}_+ \times \mathcal{O}
\end{equation}
with the Dirichlet boundary condition $u|_{\mathbb{R}_+ \times \partial \mathcal{O}} = 0$ and the initial condition $u(0,.)=X_0$, where $\mathcal{O} \subset \mathbb{R}^d$ is a bounded domain with smooth boundary $\partial \mathcal{O}$, $\alpha >0$ is a parameter (e.g., the heat conductivity), $y\in \mathcal{O}$ is fixed, $\delta_y$ stands for the Dirac distribution at the point $y$ and $(\eta^H(t))$ is formally one-dimensional fractional noise with the Hurst parameter $H\in(0,1)$.\\

The equation \eqref{eq: heat eq point noise - naive} may be treated rigorously as the system
\begin{equation}\label{eq:heat eq point noise - rigor}
    dX(t) = \alpha AX(t) dt + \Phi dB^H(t), \quad t>0,
\end{equation}
\begin{equation}\label{eq:heat eq point noise - init cond}
    X(0)=X_0
\end{equation}
in a standard way in the state space $\mathcal{V} = L^2(\mathcal{O})$, where $A=\Delta|_{Dom(A)}$, $Dom(A) = W^{2}_{2}(\mathcal{O})  \cap W^{1}_{2,0}(\mathcal{O})$ is the Dirichlet Laplacian, $(B^H(t),t\in \mathbb{R})$ is the fractional Brownian motion in $U = \mathbb{R}$ and $\Phi = \delta_y$. It is well known that the operator $A$ is strictly negative and generates exponentially stable analytic semigroup (cf. \eqref{eq: A2}). To verify \eqref{eq: A0} (where we put $\beta = 0)$ note that $\Phi \in \mathcal{L}(\mathbb{R}, (\mathcal{C}(\mathcal{O}))') \cong (\mathcal{C}(\mathcal{O}))'$ and $Dom ((-A)^{1-\epsilon}) \subset W^{2(1-\epsilon)}_{2}(\mathcal{O})$ for each $\epsilon \in (0,1)$. By Sobolev embedding theorem we have
\[
(\mathcal{C}(\mathcal{O}))' \hookrightarrow (W^{2(1-\epsilon)}_{2}(\mathcal{O}))' \hookrightarrow (Dom ((-A)^{1-\epsilon}))',
\]
if $0<\epsilon<1-\frac{d}{4}$ which verifies \eqref{eq: A0}. Furthermore, we have that
\[
|S(t)\Phi|_{\mathcal{L}}\leq |S(t)|_{\mathcal{L}(D_A^{\epsilon-1},\mathcal{V})}  \cdot |(-A)^{\epsilon -1} \Phi|_{\mathcal{L}} \leq \frac{c_T}{t^{1-\epsilon}}, \quad t \in (0,T], T>0,
\]

and since $U=\mathbb{R}$ operator norm and Hilbert-Schmidt norm of $S(t)\Phi$ are equivalent. Hence, the condition \eqref{eq: A1} is satisfied if $\gamma = 1-\epsilon < H$. Therefore all conditions \eqref{eq: A0}-\eqref{eq: A2} are satisfied by a choice of $\epsilon$,
\[
1-\frac{d}{4} > \epsilon > 1-H,
\]
which is possible if $H>\frac{d}{4}$. Note that this restriction is imposed because in this Example we consider only solution with values in a function space (more precisely, the state space $\mathcal{V}$ is $L^2(\mathcal{O})$). If it is omitted the solution will still exist in a suitable distribution space.\\

Given $w \in \mathcal{V} = L^2(\mathcal{O})$, let us examine the condition $\langle Q_{\infty} w,w\rangle_\mathcal{V} > 0$ that is needed to define the estimators \eqref{eq: proj_est_continuous} and \eqref{eq: proj_est_discrete}. The initial value $\tilde{X}$ for the equation \eqref{eq: SPDE} (with $\alpha = 1$), the law of which is the stationary one, may be defined as
\[
\tilde{X} = \int_{-\infty}^{0}S(-t)\Phi dB^{H}(t)
\]
(see e.g. \cite{MaPo-Ergo2}). Therefore, by reflexivity of increments of the process $(B^{H}(t), t\geq 0)$ (cf. \cite{Co}) we have that
\[
Law(\tilde{X}) = Law\left(\int_{0}^{\infty}S(t)\Phi dB^{H}(t)\right) = N(0,Q_{\infty}),
\]
hence
\[
\langle Q_{\infty} w,w\rangle_\mathcal{V} = \mathbb{E} \left\langle \int_{0}^{\infty}S(t)\Phi dB^{H}(t), w\right\rangle^2_\mathcal{V} = \mathbb{E} \left( \int_{0}^{\infty} \langle \Phi^{*} S^{*}(t) w, dB^{H}(t) \rangle_\mathcal{V} \right)^2.
\]
For simplicity, assume that $H > \frac{1}{2}$ (for $H < \frac{1}{2}$ the argument is similar). Computing the variance on the r.h.s. we obtain
\begin{equation}\label{eq: Qww by K}
\langle Q_{\infty} w,w\rangle_\mathcal{V} = \int_{0}^{\infty} \left|\mathcal{K}^{*}\left(\Phi^{*} S^{*}(.) w\right) (r)\right|^2 dr,
\end{equation}
where
\[
\mathcal{K}^{*}(f)(r) = r^{\frac{1}{2}-H}I_{-}^{H-\frac{1}{2}}\left(u^{H-\frac{1}{2}}f(u)\right)(r), \quad r>0,
\]
is defined for $f:\mathbb{R}_{+} \to \mathbb{R}$ such that $\int_{0}^{\infty}|\mathcal{K}^{*}(f)(r)|^2 dr < \infty$ and $I_{-}^{\alpha}$ denotes the fractional integral
\[
(I_{-}^{\alpha} f)(r) = \Gamma(\alpha)^{-1}\int_{\mathbb{R}}f(u)(u-r)^{\alpha -1}_{+} du, \quad \alpha \in (0,1),
\]
(cf. \cite{PT} for details). Assume that $\langle Q_{\infty} w,w\rangle_\mathcal{V} =0$ for some $w \in \mathcal{V}$. By \eqref{eq: Qww by K} and Lemma 6.1 in \cite{PT} we have that $\Phi^{*} S^{*}(t) w = 0$ for $t>0$. On the other hand,
\[
\Phi^{*} S^{*}(t) w = \int_{\mathcal{O}}G(t,y,\xi) w(\xi) d\xi,
\]
where $G$ is the Green function corresponding to Dirichlet Laplacian on the domain $\mathcal{O}$. Therefore if we assume that $w \geq 0$ a.e. and $w > 0$ on a set of positive Lebesgue measure, by strict positivity of the Green function $G$ on the domain $\mathbb{R}_{+} \times \mathcal{O} \times \mathcal{O}$ we obtain $\Phi^{*} S^{*}(t) w \neq 0$ and by contradiction we get $\langle Q_{\infty} w,w\rangle_\mathcal{V} > 0$, the condition needed in \eqref{eq: proj_est_continuous} and \eqref{eq: proj_est_discrete}. \\

For instance, if $w$ is a characteristic function (indicator) of a subset of $\mathcal{O}$ of positive Lebesgue measure, we may consider the estimators \eqref{eq: proj_est_continuous} and \eqref{eq: proj_est_discrete} to be built upon a partial observation of the process "through a window".\\

On the other hand, if for example $\mathcal{O} = (0,1)$ and $w(\xi) = \sin 4 \pi \xi$, we have that
\[
S^{*}(t) w(\xi) = S(t) w(\xi)= e^{-16 \pi^2 t} \sin 4 \pi \xi,
\]
so choosing $y = \frac{1}{2}$ we obtain $\Phi^{*} S^{*}(t) w = 0$ and the estimators \eqref{eq: proj_est_continuous} and \eqref{eq: proj_est_discrete} cannot be defined.\\

Now, Theorem \ref{cor: consistency} provides the strong consistency of the estimators \eqref{eq: MC_est_continuous}, \eqref{eq: MC_est_discrete}, \eqref{eq: proj_est_continuous} and \eqref{eq: proj_est_discrete}. Moreover, if  $\frac{d}{4}< H< \frac{3}{4}$, Corollary \ref{cor: CLT} implies asymptotic normality of the estimators with the upper bounds for the speed of the convergence given in Theorem \ref{thm: B-E bounds for estimators - stationary} and Remark \ref{rem: B-E bounds for estimators - non-stationary}.

\section{Appendix}

\begin{Lemma}\label{lem: 4  cumul bound infty D}
Consider the sequence $F_n$ defined in \eqref{eq: def V_n F_n - infty}. The $3^{rd}$ and $4^{th}$ cumulants thereof satisfy the following bounds:

\[
\kappa_3(F_n) \leq \frac{C_1}{\sqrt{n}\; s_n^{3/2}} \biggl( \sum_{|i|<n} |Q(i)|^{3/2}_{\mathcal{L}_2} \biggr)^2,
\]
\[
\kappa_4(F_n) \leq \frac{C_2}{n\; s_n^2} \biggl( \sum_{|i|<n} |Q(i)|^{4/3}_{\mathcal{L}_2} \biggr)^3.
\]
\end{Lemma}

\begin{proof}
To avoid technical difficulties with infinite sums, we shall work with projections to finite-dimensional subspaces and then pass to the limit. Recall the setting in the subsection \ref{subsec: infty D sequences}. For $N < \infty$ define $\mathcal{V}^{(N)}$ as the subspace of $\mathcal{V}$ spanned by its first $N$ orthonormal vectors  $\{e_{k}\}_{k=1}^{N}$. Consider the projections of the stationary Gaussian sequence $Z_i$ onto $\mathcal{V}^{(N)}$:
\[
Z_i^{(N)} := \sum_{k=1}^{N}\langle Z_i,e_k \rangle_{\mathcal{V}} \; e_k.
\]
Clearly, $(Z_i^{(N)}: i \in \mathbb{Z})$ is a stationary $\mathcal{V}^{(N)}$-valued centered Gaussian sequence with a covariance operator $Q^{(N)}$ and auto-covariance operators $Q^{(N)}(i-j)$ being restrictions of $Q$ and $Q(i-j)$ onto $\mathcal{V}^{(N)}$ in the following sense:
\[
\langle Q^{(N)} u, v \rangle_{\mathcal{V}} = \langle Q u, v \rangle_{\mathcal{V}}, \quad \langle Q^{(N)}(i-j) u, v \rangle_{\mathcal{V}} = \langle Q(i-j) u, v \rangle_{\mathcal{V}}, \quad \forall u,v \in \mathcal{V}^{(N)}.
\]
Following the approach in subsection \ref{subsec: infty D sequences}, for each $N$ we define an appropriate isonormal Gaussian process $\mathbb{X}^{(N)}$ on suitable Hilbert space $\mathcal{H}^{(N)}$   corresponding to the stationary sequence $(Z_i^{(N)}: i \in \mathbb{Z})$ and consider the sequences
\[
V_{n}^{(N)} :=  \frac{1}{\sqrt{n}} \sum_{i=1}^{n}(|Z^{(N)}_i|_{\mathcal{V}^{(N)}}^2 - {\rm Tr}(Q^{(N)})), \quad s^{(N)}_n := \mathbb{E}(V^{(N)}_n)^2, \quad F^{(N)}_n := \frac{V^{(N)}_n}{\sqrt{s^{(N)}_n}}.
\]
Rewrite $F^{(N)}_n$ as in \eqref{eq: Fn via I_2}:
\begin{equation*}
\begin{aligned}
&F^{(N)}_n = \frac{1}{\sqrt{n \; s^{(N)}_n}}\sum_{i=1}^{n} \sum_{k=1}^{N} (\langle Z^{(N)}_i,e_k\rangle_{\mathcal{V}^{(N)}}^2 - \langle Q e_k,e_k\rangle_{\mathcal{V}^{(N)}})\\
& = I_2\biggl( \frac{1}{\sqrt{n \; s^{(N)}_n}} \sum_{i=1}^{n} \sum_{k=1}^{N} h_{i,e_k}^{\otimes 2} \biggr) = I_2(f^{(N)}_n).
\end{aligned}
\end{equation*}

Recall that the $4^{th}$ cumulant of $F^{(N)}_n$ can be bounded above by the norm of the tensor contraction of $f^{(N)}_n$:
\begin{equation}\label{eq: cumul and contr}
\kappa_4(F^{(N)}_n) \leq C \; |f^{(N)}_n \otimes_1 f^{(N)}_n|^2_{(\mathcal{H}^{(N)})^{\otimes 2}},
\end{equation}
where $C$ is a universal constant (see \cite{Bierme_et_al 2012}).
The tensor contraction of order 1, denoted be $\otimes_1$, is defined as follows:
\[
f \otimes_1 g := \sum_{k=1}^{\infty} \langle f,h_k\rangle \otimes \langle g,h_k\rangle  \; \in (\mathcal{H}^{(N)})^{\otimes 2},\quad \forall f,g \in (\mathcal{H}^{(N)})^{\odot 2},
\]
where $\{h_k\}_{k=1}^{\infty}$ is an orthonormal basis of $\mathcal{H}^{(N)}$. \\
Now by the definition of $\otimes_1$ and its bilinearity, we can write
\begin{equation*}
\begin{aligned}
&f^{(N)}_n \otimes_1 f^{(N)}_n = \frac{1}{n \; s^{(N)}_n}  \sum_{i,j=1}^{n} \sum_{k,l=1}^{N} h_{i,e_k}^{\otimes 2} \otimes_1 h_{j,e_l}^{\otimes 2}\\
&= \frac{1}{n \; s^{(N)}_n}  \sum_{i,j=1}^{n} \sum_{k,l=1}^{N} \langle h_{i,e_k},h_{j,e_l}\rangle_{\mathcal{H}^{(N)}} \; (h_{i,e_k} \otimes h_{j,e_l}),
\end{aligned}
\end{equation*}
and, consequently, using the fact that
\[
\langle f_1 \otimes f_2 , g_1 \otimes g_2 \rangle_{(\mathcal{H}^{(N)})^{\otimes 2}} = \langle f_1,g_1\rangle_{\mathcal{H}^{(N)}} \; \langle f_2,g_2\rangle_{\mathcal{H}^{(N)}},
\]
and the relation for adjoint operators $(Q^{(N)}(i-j))^{*} =  Q^{(N)}(j-i)$, we calculate further:
\begin{equation*}
\begin{aligned}
&|f^{(N)}_n \otimes_1 f^{(N)}_n|^2_{(\mathcal{H}^{(N)})^{\otimes 2}} = \biggl( \frac{1}{n \; s^{(N)}_n}\biggr)^2  \sum_{i,i'=1}^{n}\sum_{j,j'=1}^{n} \sum_{k,k'=1}^{N} \sum_{l,l'=1}^{N} \langle h_{i,e_k},h_{j,e_l}\rangle_{\mathcal{H}^{(N)}} \langle h_{i',e_{k'}},h_{j',e_{l'}}\rangle_{\mathcal{H}^{(N)}}\\
&\langle h_{i,e_k},h_{i',e_{k'}}\rangle_{\mathcal{H}^{(N)}} \langle h_{j,e_l},h_{j',e_{l'}}\rangle_{\mathcal{H}^{(N)}} \\
&= \biggl( \frac{1}{n \; s^{(N)}_n}\biggr)^2  \sum_{i,i'=1}^{n}\sum_{j,j'=1}^{n} \sum_{k,k'=1}^{N} \sum_{l,l'=1}^{N} \langle Q^{(N)}(i-j) e_k,e_l \rangle_{\mathcal{V}^{(N)}} \langle Q^{(N)}(i'-j') e_{k'},e_{l'}\rangle_{\mathcal{V}^{(N)}}\\
&\langle Q^{(N)}(i-i')e_k,e_{k'}\rangle_{\mathcal{V}^{(N)}} \langle Q^{(N)}(j-j')e_l,e_{l'}\rangle_{\mathcal{V}^{(N)}}\\
& = \biggl( \frac{1}{n \; s^{(N)}_n}\biggr)^2  \sum_{i,i'=1}^{n}\sum_{j,j'=1}^{n} \sum_{k'=1}^{N} \sum_{l=1}^{N} \langle Q^{(N)}(i'-j') e_{k'},Q^{(N)}(j-j')e_l\rangle_{\mathcal{V}^{(N)}}\\
&\langle Q^{(N)}(j-i) e_l,  Q^{(N)}(i'-i)e_{k'}\rangle_{\mathcal{V}^{(N)}}  \\
& = \biggl( \frac{1}{n \; s^{(N)}_n}\biggr)^2  \sum_{i,i'=1}^{n}\sum_{j,j'=1}^{n} \sum_{k'=1}^{N} \langle Q^{(N)}(j'-j)Q^{(N)}(i'-j') e_{k'}, Q^{(N)}(i-j) Q^{(N)}(i'-i)e_{k'}\rangle_{\mathcal{V}^{(N)}} \\
& = \biggl( \frac{1}{n \; s^{(N)}_n}\biggr)^2  \sum_{i,i'=1}^{n}\sum_{j,j'=1}^{n} {\rm Tr}\biggl( Q^{(N)}(i-i') Q^{(N)}(j-i) Q^{(N)}(j'-j)Q^{(N)}(i'-j') \biggr)  \\
& \leq \biggl( \frac{1}{n \; s^{(N)}_n}\biggr)^2  \sum_{i,i'=1}^{n}\sum_{j,j'=1}^{n} |Q^{(N)}(i-i')|_{\mathcal{L}_2} \; |Q^{(N)}(j-i)|_{\mathcal{L}_2} \; |Q^{(N)}(j'-j)|_{\mathcal{L}_2} \; |Q^{(N)}(i'-j')|_{\mathcal{L}_2}.
\end{aligned}
\end{equation*}
If we denote $\rho(i):= |Q^{(N)}(i)|_{\mathcal{L}_2}$ for $i \in \mathbb{Z}$, we have $\rho(i) = \rho(-i)$ and
\[
|f^{(N)}_n \otimes_1 f^{(N)}_n|^2_{(\mathcal{H}^{(N)})^{\otimes 2}} \leq \biggl( \frac{1}{n \; s^{(N)}_n}\biggr)^2  \sum_{i,i'=1}^{n}\sum_{j,j'=1}^{n} \rho(i-i') \; \rho(j-i)\; \rho(j'-j) \; \rho(i'-j').
\]
Now we apply the method used in the proof of Proposition 6.4 in \cite{Bierme_et_al 2012} based on rewriting the last sums as convolutions and then applying the Young inequality. This leads to the inequality

\[
|f^{(N)}_n \otimes_1 f^{(N)}_n|^2_{(\mathcal{H}^{(N)})^{\otimes 2}} \leq \frac{1}{n \; (s^{(N)}_n)^2}\biggl( \sum_{|i|<n} \rho(i)^{4/3}\biggr)^3.
\]
Combining this with \eqref{eq: cumul and contr} and the fact that $|Q^{(N)}(i)|_{\mathcal{L}_2} \leq |Q(i)|_{\mathcal{L}_2}$ we obtain the bound
\[
\kappa_4(F^{(N)}_n) \leq C \; \frac{1}{n \; (s^{(N)}_n)^2} \biggl( \sum_{|i|<n} |Q(i)|^{4/3}_{\mathcal{L}_2} \biggr)^3,
\]
with a universal constant $C$. We can conclude the proof by noting that $\underset{N \to \infty}{\text{$L^{2}$-lim}} V^{(N)}_n = V_n$ implies \[\lim_{N\to \infty} s^{(N)}_n = s_n\]
and
\[\kappa_4(F_n) =  \lim_{N\to \infty} \kappa_4(F^{(N)}_n) \leq \frac{C}{n \; (s_n)^2} \biggl( \sum_{|i|<n} |Q(i)|^{4/3}_{\mathcal{L}_2} \biggr)^3.
\]

The proof of the bound for the $3^{rd}$ cumulant runs similarly. With reference to \cite{Bierme_et_al 2012} we start from the following equality
\[
\kappa_3(F^{(N)}_n) = 8 \langle f^{(N)}_n, f^{(N)}_n \otimes_1 f^{(N)}_n \rangle_{(\mathcal{H}^{(N)})^{\otimes 2}}.
\]
Using the expressions for $f^{(N)}_n$ and $f^{(N)}_n \otimes_1 f^{(N)}_n $ derived above we can continue the calculation similarly to the $4^{th}$ cumulant case:

\begin{equation*}
\begin{aligned}
&\kappa_3(F^{(N)}_n) = \frac{8}{(n\; s^{(N)}_n)^{3/2}} \sum_{i,i'=1}^{n}\sum_{j=1}^{n} \sum_{k,k'=1}^{N} \sum_{l=1}^{N} \langle h_{i,e_{k}},h_{j,e_{l}}\rangle_{\mathcal{H}^{(N)}} \langle h_{i',e_{k'}}^{\otimes 2}, h_{i,e_{k}} \otimes h_{j,e_{l}} \rangle_{(\mathcal{H}^{(N)})^{\otimes 2}}\\
&=\frac{8}{(n\; s^{(N)}_n)^{3/2}} \sum_{i,i'=1}^{n}\sum_{j=1}^{n} \sum_{k,k'=1}^{N} \sum_{l=1}^{N} \langle h_{i,e_{k}},h_{j,e_{l}}\rangle_{\mathcal{H}^{(N)}} \langle h_{i',e_{k'}}, h_{i,e_{k}}\rangle_{\mathcal{H}^{(N)}} \langle h_{i',e_{k'}}, h_{j,e_{l}}\rangle_{\mathcal{H}^{(N)}} \\
&\leq \frac{8}{(n\; s^{(N)}_n)^{3/2}} \sum_{i,i'=1}^{n}\sum_{j=1}^{n} |Q^{(N)}(i-i')|_{\mathcal{L}_2} \; |Q^{(N)}(j-i)|_{\mathcal{L}_2} \; |Q^{(N)}(i'-j)|_{\mathcal{L}_2}.
\end{aligned}
\end{equation*}
If we denote again $\rho(i):= |Q^{(N)}(i)|_{\mathcal{L}_2}$ for $i \in \mathbb{Z}$, we can follow the corresponding calculations in \cite{Bierme_et_al 2012} (based on the Young inequality for convolutions) and pass to the limit $N \to \infty$ to obtain the required bound.
\end{proof}

\begin{Lemma}\label{lem: integral for Q}
Consider constants $\rho >0, H \in (0,1)$ and $\gamma \in [0,H)$. Then for sufficiently large $t>0$:
\[
\int_{0}^{t} \int_{-\infty}^{0} e^{-\rho (-r)} (-r)^{-\gamma} \; e^{-\rho (t-s)} (t-s)^{-\gamma} \; H(2H-1)(s-r)^{2H-2}drds \leq C t^{2H-2}.
\]
\end{Lemma}
\begin{proof}
Fix $\delta \in (1/\rho , t)$ and divide the range of integration into four disjoint areas:
\begin{equation*}
\begin{aligned}
&\int_{0}^{t} \int_{-\infty}^{0} e^{-\rho (-r)} (-r)^{-\gamma} \; e^{-\rho (t-s)} (t-s)^{-\gamma} \; H(2H-1)(s-r)^{2H-2}drds \\
& = \int_{0}^{1/\rho} \int_{-\infty}^{0} ... drds +
\int_{1/\rho}^{t-\delta} \int_{-\infty}^{-\delta} ... drds +
\int_{1/\rho}^{t-\delta} \int_{-\delta}^{0} ... drds \\
&+\int_{t-\delta}^{t} \int_{-\infty}^{0} ... drds.
\end{aligned}
\end{equation*}
We shall treat each integral separately. Here $C$ stands for a positive constant (independent of $t$), which may change from line to line.
\begin{equation*}
\begin{aligned}
&\int_{0}^{1/\rho} \int_{-\infty}^{0} e^{-\rho (-r)} (-r)^{-\gamma} \; e^{-\rho (t-s)} (t-s)^{-\gamma} \; H(2H-1)(s-r)^{2H-2}drds  \\
&\leq \frac{C e^{-\rho (t-1/\rho)}}{(t-1/\rho)^\gamma}\int_{0}^{1/\rho} \int_{-\infty}^{0} e^{-\rho (-r)} (-r)^{-\gamma}(s-r)^{2H-2}drds \leq C e^{-\rho t}.
\end{aligned}
\end{equation*}
The second integral (after slight modification) can be treated as the corresponding term in the proof of Theorem 2.3 in \cite{Cheridito_et_al 2003}, which justifies the last inequality in the following calculations:
\begin{equation*}
\begin{aligned}
&\int_{1/\rho}^{t-\delta} \int_{-\infty}^{-\delta}  e^{-\rho (-r)} (-r)^{-\gamma} \; e^{-\rho (t-s)} (t-s)^{-\gamma} \; H(2H-1)(s-r)^{2H-2}drds  \\
&\leq C \int_{1/\rho}^{t-\delta} \int_{-\infty}^{-\delta}  e^{-\rho (-r)} \; e^{-\rho (t-s)} \; (s-r)^{2H-2}drds \\
&\leq C \int_{1/\rho}^{t} \int_{-\infty}^{0}  e^{-\rho (-r)} \; e^{-\rho (t-s)} \; (s-r)^{2H-2}drds \leq
C t^{2H-2}.
\end{aligned}
\end{equation*}
For the third integral, note that
\begin{equation*}
\begin{aligned}
&\int_{1/\rho}^{t-\delta} \int_{-\delta}^{0}  e^{-\rho (-r)} (-r)^{-\gamma} \; e^{-\rho (t-s)} (t-s)^{-\gamma} \; H(2H-1)(s-r)^{2H-2}drds  \\
&\leq C \int_{1/\rho}^{t-\delta}  e^{-\rho (t-s)}s^{2H-2}  \left(\int_{-\delta}^{0}(-r)^{-\gamma}dr \right) ds \leq C \int_{1/\rho}^{t-\delta}  e^{-\rho (t-s)}s^{2H-2} ds.
\end{aligned}
\end{equation*}
To see that this can be bounded above by $C t^{2H-2}$, we can observe that
\begin{equation*}
\begin{aligned}
&\frac{\int_{1/\rho}^{t-\delta}  e^{-\rho (t-s)}s^{2H-2} ds}{t^{2H-2}} \leq \int_{1/\rho}^{t}  e^{-\rho (t-s)}\left(\frac{s}{t}\right)^{2H-2} ds = \int_{0}^{t-1/\rho}  e^{-\rho y}\left(1- \frac{y}{t}\right)^{2H-2} dy \\
& = \int_{0}^{\infty}  e^{-\rho y}\left(1- \frac{y}{t}\right)^{2H-2} \mathbb{I}_{[0,t-1/\rho]}(y) dy \overset{t\to \infty}{\longrightarrow} \int_{0}^{\infty}  e^{-\rho y} dy < \infty,
\end{aligned}
\end{equation*}
where $\mathbb{I}$ denotes the indicator function and the last limit follows from the Dominated Convergence Theorem with integrable majorant
\[
e^{-\rho y}\left(1- \frac{y}{t}\right)^{2H-2} \mathbb{I}_{[0,t-1/\rho]}(y) \leq e^{-\rho y}\left(1- \frac{y}{y+1/\rho}\right)^{2H-2} = (1/\rho)^{2H-2} e^{-\rho y} \left(y+ 1/\rho\right)^{2-2H}.
\]
Now let us find the upper bound for the fourth integral:
\begin{equation*}
\begin{aligned}
&\int_{t-\delta}^{t} \int_{-\infty}^{0} e^{-\rho (-r)} (-r)^{-\gamma} \; e^{-\rho (t-s)} (t-s)^{-\gamma} \; H(2H-1)(s-r)^{2H-2}drds  \\
&\leq C (t-\delta)^{2H-2} \int_{t-\delta}^{t} \int_{-\infty}^{0}  e^{-\rho (-r)} (-r)^{-\gamma} \; (t-s)^{-\gamma} \; drds \\
& = C (t-\delta)^{2H-2} \left(\int_{0}^{\delta} u^{-\gamma} du \right) \left( \int_{0}^{\infty}  e^{-\rho v} v^{-\gamma} dv \right) \leq C t^{2H-2},
\end{aligned}
\end{equation*}
which completes the proof.
\end{proof}

\end{document}